\documentclass[12pt, reqno]{article}
\usepackage[utf8]{inputenc}
\usepackage{ dsfont }
\usepackage{amsfonts}
\usepackage{amsmath}
\usepackage{commath}
\usepackage{amsthm}
\usepackage{ amssymb }
\usepackage{enumerate} 
\usepackage{mathtools}
\usepackage{verbatim}
\usepackage{bm}
\allowdisplaybreaks

\usepackage[usenames]{color}
\usepackage{amssymb}
\usepackage{amsmath}
\usepackage{amsthm}
\usepackage{amsfonts}
\usepackage{amscd}
\usepackage{graphicx}

\usepackage[colorlinks=true,
linkcolor=webgreen,
filecolor=webbrown,
citecolor=webgreen]{hyperref}

\definecolor{webgreen}{rgb}{0,.5,0}
\definecolor{webbrown}{rgb}{.6,0,0}

\usepackage{color}
\usepackage{fullpage}
\usepackage{float}

\newcommand{\seqnum}[1]{\href{https://oeis.org/#1}{\underline{#1}}}

\begin{document}
	
	\theoremstyle{plain}
	\newtheorem{theorem}{Theorem}
	\newtheorem{corollary}[theorem]{Corollary}
	\newtheorem{lemma}[theorem]{Lemma}
	\newtheorem{proposition}[theorem]{Proposition}
	
	\theoremstyle{definition}
	\newtheorem{definition}[theorem]{Definition}
	\newtheorem{example}[theorem]{Example}
	\newtheorem{conjecture}[theorem]{Conjecture}
	
	\theoremstyle{remark}
	\newtheorem{remark}[theorem]{Remark}
	
	\begin{center}
	\end{center}
	
	\begin{center}
		\vskip 1cm{\LARGE\bf  On the Erd\H{o}s-Sloane and Shifted Sloane Persistence Problems\let\thefootnote\relax\footnote{This work has been supported by ``Projeto Tem\'atico Din\^amica em Baixas Dimens\~oes'' FAPESP Grant  2016/25053-8
			}\\
			\vskip 1cm}

		\large
		
		Gabriel Bonuccelli\\
		Lucas Colucci\\ 
		Edson de Faria\\

		\vspace{0.2cm}
		Instituto de Matem\'atica e Estat\'istica\\
		Universidade de S\~ao Paulo\\
		S\~ao Paulo\\ 
		Brazil\\
		\href{mailto:gabrielbhl@ime.usp.br}{\tt \{gabrielbhl,lcolucci,edson\}@ime.usp.br}
	\end{center}

	\vskip .2 in

	\begin{abstract}
		
		In this paper, we investigate two variations on the so-called persistence problem of Sloane: the shifted version, which was introduced by Wagstaff; and the nonzero version, proposed by Erd\H{o}s. We explore connections between these problems and a recent conjecture of de Faria and Tresser regarding equidistribution of the digits of some integer sequences and a natural generalization of it.
		
	\end{abstract}
	
	\section{Introduction}\label{intro}
	
	In 1973, Sloane \cite{sloane1973persistence} proposed the following question: given a positive integer $n$, multiply all its digits together to get a new number, and keep repeating this operation until a single-digit number is obtained. The number of operations needed is called the \emph{persistence} of $n$. Is it true that there is an absolute constant $C(b)$ (depending solely on the base $b$ in which the numbers are written) such that the persistence of every positive integer is bounded by $C(b)$? Despite many computational searches, heuristic arguments and related conjectures in favor of a positive answer, no proof or disproof of this conjecture has been found so far.

	Many variants of Sloane's problem have been considered as well. The famous book of Guy \cite{guy2013unsolved} mentions that Erd\H{o}s introduced the version of the Sloane problem wherein only the nonzero digits of a number are multiplied in each iteration, which we call the \emph{Erd\H{o}s-Sloane problem}. Another variant was raised by Diamond and Reidpath \cite{diamond1998counterexample}, where instead of the usual base $b$, numbers are taken to the so-called factorial base. Less related variations include the additive persistence (when the digits are summed up instead of multiplied) \cite{hinden1974additive} or versions where even more general functions of the digits are considered \cite{beardon1998sums, herzberg2014some}.
	
	In this paper, we will be concerned with the Erd\H{o}s-Sloane version and with the shifted Sloane problem, introduced by Wagstaff \cite{wagstaff1981iterating}, which consists in shifting all the digits of the number by a fixed positive integer before multiplying them.
	
	\section{Definitions and notation}\label{sec:def}
	
	For integers $t \geq 0$ and $b \geq 2$, the \emph{$t$-shifted Sloane map} in base $b$ is the map $S_{t,b}$ that takes a nonnegative integer $n = \sum_{i=0}^k d_ib^i$ (as usual, $0 \leq d_i \leq k-1$ for all $i$ and $d_k > 0$) to the integer $S_{t,b}(n)=\prod_{i=0}^k(d_i+t)$. This function was introduced (in the special case $b=10$) by Wagstaff \cite{wagstaff1981iterating}, motivated by a question of Erd\H{o}s and Kiss. Note that $t=0$ corresponds to the map that we iterate in the original persistence problem. Furthermore, the \emph{Erd\H{o}s-Sloane map} in base $b$, denoted by $S_b^*$, is the map that sends $n = \sum_{i=0}^k d_ib^i$ to $S_b^*(n)=\prod_{0 \leq i \leq k, d_i \neq 0 }d_i$. The set of nonnegative integers is denoted by $\mathbb{N}$. We use the notation $f^k(n)$ to denote the $k$-th iterate of a map $f$ on the point $n$, i.e., $f^0(n)=n$ and $f^{k}(n)=f(f^{k-1}(n))$ for every $k \geq 1$, and we let $\mathrm{Per}(f)$ denote the set of periodic points of $f$. Finally, for $0 \leq d \leq b-1$ and an integer $n$, we let \emph{$\#d(n)_b$} denote the number of digits $d$ in the expansion of $n$ in base $b$ (e.g., $\#2(100_{10})_3 = 1$, since $100_{10} = 10201_3$), and $\#(n)_b$ denote the number of digits of the base-$b$ expansion of $n$. The subscripts $b$ will be omitted when clear from the context.
	
	We are interested in the dynamics of $S_{t,b}$ and $S_b^*$. Contrarily to the original problem ($t=0$) and to $S_b^*$, for $t \geq 1$ it is not even clear that every $n$ reaches a fixed point or a cycle of $S_{t,b}$. If it does, the smallest number of iterations to reach it will be called, as usual, the \emph{persistence} of $n$, and it will be denoted, respectively, by $\nu_{t,b}(n)$ and $\nu_b^*(n)$ (we set those values to $\infty$ in case the corresponding sequence of iterates diverges). Even the basic question of whether $\nu_{t,b}(n)$ is finite for every $n$ is not so readily answered for many values of $t$ and $b$, and it is open for most of them. On the other hand, it is easy to see that $S_{b,b}(n) > n$ holds for every $b$ and $n$, so $\nu_{t,b}(n) = \infty$ for every $n$ whenever $t \geq b$. Thus, from now on, we will assume that $t \leq b-1$ unless stated otherwise.    
	
	\section{Questions}\label{sec:questions}
	
	For every $b \geq 2$ and $0 \leq t \leq b-1$, one defines the $t$-shifted Sloane problem in base $b$ as in the previous section. We will refer to this problem as the $(t,b)$ problem for short. Furthermore, for every $b \geq 2$, we may consider the Erd\H{o}s-Sloane problem in base $b$, and, similarly, we will refer to this problem as the $(*,b)$ problem.
	
	For each of the $(t,b)$ and $(*,b)$ problems, there are different questions one may ask about the corresponding iterating map $f=S_{t,b}$ (or $f=S_b^*$).	
	\begin{enumerate}
		
		\item The most basic question is the following: let $A_f$ denote the set of nonnegative integers $n$ such that the sequence of iterates $(f^k(n))_{k \geq 0}$ stabilizes (i.e., reaches either a cycle or a fixed point of $f$). What can be said about $A_f$? It is trivial that $A_f=\mathbb{N}$ in the original problem ($t=0$) and in the Erd\H{o}s-Sloane problem. We will prove (Theorem \ref{thm:f1bconverges}) that $A_f=\mathbb{N}$ for $t=1$ and $b \geq 2$ as well (extending a result of Wagstaff \cite{wagstaff1981iterating}). Furthermore, we will prove that some natural conjectures on the equidistribution of digits of some sequences imply that $A_f=\mathbb{N}$ for some pairs $(t,b)$ and that $\mathbb{N}-A_f$ contains all sufficiently large integers for other pairs $(t,b)$, but our results do not cover all the range of $(t,b)$ (Theorems \ref{thm:f23converges}, \ref{thm:45diverges}, \ref{thm:b/4converges}, \ref{thm:tsmallconverges} and \ref{thm:divergencelarge}).
		
		\item In case $A_f=\mathbb{N}$, i.e., every $n \in \mathbb{N}$ stabilizes under iteration by $f$, is there a universal constant that bounds the persistence of all numbers, that is, is there $C$ such that $\nu_{t,b}(n) \leq C$ (or $\nu_{b}^*(n) \leq C$) for every $n \in \mathbb{N}$? Note that the positive answer to the question for $t=0$ is the original Sloane conjecture. Perhaps surprisingly, we prove that the equidistribution conjectures imply a negative answer for $t=1$ and for the Erd\H{o}s-Sloane problem, which shows a pronounced difference in the behaviors of the $(0,b)$ problem and the $(1,b)$ and $(*,b)$ problems (Theorems \ref{thm:erdosbase3}, \ref{thm:perserdos}, \ref{thm:t1b3pers} and \ref{thm:perst1}).
		
		\item Still assuming $A_f=\mathbb{N}$, we know that every integer $n$ reaches either a fixed point or a cycle of $f$. What are the cycles and the fixed points of $f$? Besides the trivial cases $(0,b)$ and $(*,b)$, we are able to describe them precisely in the case $t=1$ for any $b \geq 3$ (Theorem \ref{thm:f1bconverges}).
		
		\item Finally, the most detailed question we deal with is the following: for a cycle (or fixed point) $C$ of $f$, which integers $n$ tend to an element of $C$ under iteration by $f$? That is, what are the backward orbits of each cycle $C$? We can answer this question precisely only in the case $t=1$ and $b=3$ (Theorem \ref{t1b3cycles}).
		
	\end{enumerate}
	
	\section{Equidistribution of digits in products of primes}

	If an integer $n$ contains a digit zero in its base-$b$ expansion, then $S_{0,b}(n) = 0$; otherwise, it is a product of positive digits in base $b$, i.e., the integers from 1 to $b-1$. This simple fact implies that almost all integers have persistence equal to one, in the sense that the number of integers up to $N$ having this property is asymptotically equal to $N$ when $N \to \infty$. Furthermore, it implies that, when considering the dynamics of $S_{0,b}$ (or $S_b^*$), it is frequently enough to deal with products of integers less than $b$, i.e., products of power of primes smaller than $b$. Similarly, for $S_{t,b}$,  it is enough to consider products of integers between $t$ and $t+b-1$. Based on strong computational evidence and heuristic models, de Faria and Tresser \cite{de2014sloane} proposed a conjecture that states, in particular, that some sequences of this kind of numbers have a very regular asymptotic distribution of digits. Before stating their conjecture, we introduce one more definition: given $\varepsilon > 0$, a number $n$, and a base $b$, we say that the digits of $n$ are $\varepsilon$-equidistributed (in base $b$) if, for every digit $d \in \{0,\dots,b-1\}$, we have $|\frac{\#d(n)_b}{\#(n)_b}-\frac{1}{b}|<\varepsilon$.	
	\begin{conjecture}[de Faria, Tresser \cite{de2014sloane}]\label{conj:equidist}
		Given an integer $q > 1$, a finite set of primes $F$ that does not contain all the primes dividing $q$, and a positive integer $a$, let $(N_i)_{i\geq 0}$ be a sequence of integers such that $N_0 = a$ and, for every $k \geq 0$, $N_{k+1} = N_k \cdot p_k$, where $p_k \in F$. Then the digits $\{0,\dots,q-1\}$ are asymptotically equidistributed when $n \to \infty$ in the base-$q$ expansion of the $N_i$. That is, given $\varepsilon >0$, there is $n_0$ such that $N_n$ is $\varepsilon$-equidistributed in base $q$ for every $n \geq n_0$.
	\end{conjecture}
	This form of the conjecture stems from an earlier one that arose in discussions between C.\ Tresser and G.\ Hentchel \cite[p.~381]{de2014sloane}.
	\begin{remark}
		The full version of Conjecture \ref{conj:equidist} also states that the equidistribution holds for blocks of consecutive digits of any length $l > 0$, i.e., given a block of $l$ digits, its proportion in the base-$q$ expansion of the numbers in the sequences $(N_i)_{i\geq 0}$ is asymptotically equal to $\frac{1}{q^l}$.
	\end{remark}	
	\begin{remark}
		Although Conjecture \ref{conj:equidist} seems very natural, even its simplest instances are not known to be true. For instance, it is not known whether the sequence $(2^n)_{n \geq 0}$ is asymptotically equidistributed in base $3$. Indeed, even the old conjecture of Erd\H{o}s \cite{erdos1979some} that states that all but finitely many terms of this sequence contain a digit two in its ternary expansion is still open.
	\end{remark}	
	Although Conjecture \ref{conj:equidist} is enough to prove results in base $3$ (and, in some cases, base $4$), for larger bases one needs a uniform, or ``multidimensional'' generalized version, which we now state.	
	\begin{conjecture}[Uniform generalization of Conjecture \ref{conj:equidist}]\label{conj:uniformdistr}
		Let $q>1$ be an integer, $F=\{p_1,\dots,p_k\}$ be a finite set of primes that does not contain all the primes dividing $q$, 
		and $a$ be a positive integer. Then, for every $\varepsilon > 0$, 
		there exists $N$ such that $a\prod_{i=1}^k p_i^{\alpha_i}$ is $\varepsilon$-equidistributed in base $q$ whenever $\alpha_i \geq N$ for some $i \in \{1,\dots,k\}$.
	\end{conjecture}	
	
	\section{Main results}
	
	\subsection{Erd\H{o}s-Sloane problem}
	
	First, we consider the Erd\H{o}s-Sloane version. It is trivial that the only periodic points of the Erd\H{o}s-Sloane map in base $b$ are the fixed points $1,2,\dots,b-1$. As for the persistence of a number, we prove that Conjectures \ref{conj:equidist} and \ref{conj:uniformdistr} imply that the analog of Sloane's conjecture does not hold in this context.
	\begin{theorem}\label{thm:erdosbase3}
		Conjecture \ref{conj:equidist} implies that for both $S^*_3$ and $S^*_4$, there are integers with arbitrarily large persistence. Moreover, Conjecture \ref{conj:uniformdistr} implies the same result for $S_b^*$, for every $b \geq 5$.
	\end{theorem}	
	\begin{proof}
		We divide the proof into three cases: $b=3$, $b=4$ and $b \geq 5$.		
		\begin{enumerate}[(i)]
			\item $\textit{Base 3}$
			
			Let $f$ denote $S^*_3$ for ease of notation. Taking $q=3$, $F=\{2\}$ and $a=1$ in Conjecture \ref{conj:equidist}, we get the following statement: for every $\varepsilon > 0$, there is $N$ such that the proportion of digits $d$ ($d=0,1,2$) in $2^n$ is in $(1/3-\varepsilon,1/3+\varepsilon)$ whenever $n \geq N$.
			
			Take $\varepsilon = 1/6$. There is $N$ such that the proportion of digits $d$ ($d=0,1,2$) in $2^n$ is in $(1/6,1/2)$ whenever $n \geq N$. Take $m$ such that $m(1/6)^{t-1}(\log_3{2})^{t-1} \geq \max\{N,3\}$ and consider the integer $n=2^m$. We claim that $\nu_3^*(n) \geq t$. 
			
			As $m = m(1/6)^0(\log_3{2})^0 \geq \max\{N,3\}$, we know that $2^m$ has at least $1/6$ of its digits equal to 2. Thus, $f(m) = 2^{m_1}$, where $m_1 \geq (1/6)\log_3{2^m} = m(1/6)\log_3{2}$. As $m(1/6)\log_3{2} \geq \max\{N,3\}$, we know that the persistence of $m$ is at least two and that $2^{m_1}$ has at least 1/6 of its digits equal to 2. Thus, $f(f(2^m))=f(2^{m_1})=2^{m_2}$, where $m_2\geq 2^{m(1/6)^2(\log_3{2})^2}$. Inductively, we get that $f^{t-1}(2^m)=2^{m_{t-1}}$, where $m_{t-1}\geq 2^{m(1/6)^{t-1}(\log_3{2})^{t-1}} \geq 3$: indeed, if $f^{i}(2^m) = 2^{m_i}$ with $m_i \geq m(1/6)^{i}(\log_3{2})^{i}$, we have that $f^{i}(2^m)$ is $\varepsilon$-equidistributed, since $m(1/6)^{i}(\log_3{2})^{i} \geq N$, and then at least $1/6$ of its digits are equal to $2$. This implies that $f^{i+1}(2^m) = 2^{m_{i+1}}$, where $m_{i+1} \geq (1/6)\log_3{f^{i}(2^m)} \geq m(1/6)^{i+1}(\log_3{2})^{i+1}$, and completes the induction. Hence, $f^{t-1}(2^m) \geq 2^{m(1/6)^{t-1}(\log_3{2})^{t-1}} \geq 3$, which means that $n=2^m$ has persistence at least $t$.
			
			\item $\textit{Base 4}$
			
			In this case, we apply Conjecture \ref{conj:equidist} twice, with $q=4$, $F={3}$, $a=1$; and $q=4$, $F={3}$, $a=2$, respectively, to get that for every $\varepsilon > 0$ there is $N$ such that $3^m$ and $2\cdot 3^m$ are $\varepsilon$-equidistributed whenever $m \geq N$. Noting that, for every $a$ and $b$, we have $f(2^a\cdot3^b)$ is either equal to $f(3^b)$ or $f(2\cdot 3^b)$ and applying the same argument as in item (i) with $\varepsilon=1/8$, one can show that $\nu_4^*(3^m) \geq t$ if $m(1/8)^{t-1}(\log_3 4)^{t-1} \geq \max\{N,4\}$.
			
			\item $\textit{Larger bases}$
			
			For $b \geq 5$, let $F$ be the set of primes smaller than $b$ and $F_r=F-\{r\}$. We apply Conjecture \ref{conj:uniformdistr} for each prime divisor $r$ of $b$ and each proper divisor $d$ of $b$, with $F_r$ being the set of primes considered, $q=b$ and $a=d$. Note that, by Bertrand's postulate (which states that, for every $n>1$, there is a prime $p$ such that $n<p<2n$), there is some prime $q$ such that $(b-1)/2<q<b$, which means that $q$ does not divide $b$. Taking the maximum of the $N$ given in each application of the conjecture with a fixed $\varepsilon > 0$, we get the following statement: for every $\varepsilon > 0$, there is $N$ such that $d\prod_{p_i \in F_r}p_i^{\alpha_i}$ is $\varepsilon$-equidistributed whenever $\alpha_i \geq N$ for some $i \in \{1,\dots,k\}$, where $d$ is a proper divisor of $b$ and $r$ is a prime divisor of $b$.
			Moreover, note that, for every $n$, one has $f(bn)=f(n)$, since the expansion of $bn$ and $n$ in base $b$ differ only by one digit $0$.
			
			We claim that $\nu*_{b}(n) \geq t$, where $n=(\prod_{p_i \in F, p_i\nmid b}p_i)^m$ and $m$ satisfies the inequality $m(1/(2b))^{t-1}(\log_b 2)^{t-1} \geq \max\{N,b\}$, where $N$ is the integer given by the applications of Conjecture \ref{conj:uniformdistr} as above with $\varepsilon = 1/(2b)$.
			As $m \geq N$, $n$ is $\varepsilon$-equidistributed, whence, using a rough bound $(\prod_{p_i \in F, p_i\nmid b}p_i)^m \geq 2^m$, we have $f(n)=\prod_{p_i \leq b \text{ prime }}p_i^{\beta_i}$, with $\beta_i \geq (1/b-\varepsilon)\#(n)_b \geq (1/(2b))(\log_b 2)m$. The fact that $f(ba)=f(a)$ for every $a$ implies that $f(f(n))=f(d\prod_{p_i \in F_ r}p_i^{\alpha_i})$ for some proper divisor $d$ of $b$ and some prime $r$ dividing $b$. As $\alpha_i \geq \beta_i$ for every $p_i$ that does not divide $b$ (because no power of $p_i$ can be factored out into powers or divisors of $b$), we have $\alpha_i \geq (1/(2b))(\log_b 2)m \geq N$ for some $i \in \{1,\dots,k\}$. Then, the number $d\prod_{1\leq i \leq k}p_i^{\alpha_i}$ is $\varepsilon$-equidistributed, and hence we have $f(f(n))=\prod_{p_i \leq b \text{ prime }}p_i^{\beta'_i}$ with $\beta'_i \geq (1/(2b))\#(d\prod_{1\leq i \leq k}p_i^{\alpha_i})_b \geq(1/(2b))(\log_b 2)\alpha_i \geq (1/(2b))^2(\log_b 2)^2 \cdot m$ for every $i$, and the argument can be repeated. Inductively, the exponents of the $p_i$ in $f^{t-1}(n)$ are at least $(1/(2b))^{t-1}(\log_b 2)^{t-1} \cdot m$. In particular, $f^{t-1}(n) \geq 2^{(1/(2b))^{t-1}(\log_b 2)^{t-1} \cdot m} \geq 2^{b} \geq b^2$. By the choice of $m$, this number is at least $b$, so $n$ has persistence at least $t$.
		\end{enumerate}
	\end{proof}
	In other words, assuming Conjectures \ref{conj:equidist} and \ref{conj:uniformdistr}, Theorem \ref{thm:erdosbase3} states that $\limsup_{n\to\infty} \nu_b^*(n) = \infty$ for every $b \geq 3$. Our next result gives an estimate on this number which is sharp up to a constant factor.	
	\begin{theorem}\label{thm:perserdos}
		For each base $b\geq 3$, we have 
		\begin{equation}\label{upperbound}
		\limsup_{n\to\infty} \frac{\nu_b^*(n)}{\log\log{n}}\leq \frac{1}{\log\left(\alpha^{-1}\right)},
		\end{equation}
		where $\alpha=\log_{b}{(b-1)}$. Moreover, if Conjecture \ref{conj:uniformdistr} holds, then we have
		\begin{equation}\label{lowerbound}
		\limsup_{n\to\infty} \frac{\nu_b^*(n)}{\log\log{n}}\geq  \frac{1}{\log{(\beta^{-1})}},
		\end{equation}
		where $\beta=\frac{\log_{b}{2}}{2b}$.
	\end{theorem}	
	\begin{proof}
		For ease of notation, let us denote by $f$ the Erd\H{o}s-Sloane map $S_b^*$ in base $b$. The proof is naturally divided into two parts.
		\begin{enumerate}
			\item[(i)] \emph{Upper bound}. Given $n > b$, let us denote by $k_j$ ($j=0,1,\ldots$) the number of digits of $f^j(n)$ in base $b$. Note that $k_j\leq 1+ \log_{b}{f^j(n)}$ for all $j$. Since each digit in the base-$b$ expansion of $f^j(n)$ is at most $b-1$, we have $f^{j+1}(n)\leq (b-1)^{k_j}$. Therefore, for all $j\geq 0$ we get 
			\begin{equation*}
				k_{j+1} \leq 1+ k_j\log_{b}{(b-1)} = 1+\alpha k_j.
			\end{equation*}
			By induction, it follows that for all $j\geq 1$ we have
			\begin{equation}\label{geomest}
			k_j\leq \alpha^jk_0 +(1+\alpha +\alpha^2+\cdots + \alpha^{j-1}) <\alpha^jk_0 + \frac{1}{1-\alpha}.
			\end{equation}
			Since $\alpha<1$, the first term in the right-hand side of \eqref{geomest} goes to zero as $j\to\infty$. Thus, let $j_0$ be the smallest natural number such that $\alpha^{j_0}k_0<1$. An easy calculation shows that 
			\begin{equation*}
			j_0 = \left\lceil \frac{\log_{b}{k_0}}{\log_{b}{(\alpha^{-1})}} \right\rceil, 
			\end{equation*}
			and since $k_0\leq 1+\log_{b}{n}\leq 2\log_{b}{n}$ when $n\geq b$, it follows that
			\begin{equation*}
			j_0\leq \frac{\log_{b}\log_{b}{n}}{\log_{b}{(\alpha^{-1})}}+\left(1+ \frac{\log_{b}{2}}{\log_{b}{(\alpha^{-1})}}\right).
			\end{equation*}
			But now note that $k_{j_0}\leq D$, where $D=1+\lceil (1+\alpha)^{-1}\rceil$. Hence, defining
			\begin{equation*}
			M = \max\left\{ \nu_b^*(m):m\text{ has at most } D \text{ digits in base }b\right\}< \infty,
			\end{equation*}
			we see that $\nu_b^*(f^{j_0}(n))\leq M$. Since we clearly have $\nu_b^*(n)= j_0 + \nu_b^*(f^{j_0}(n))$, it follows that
			\begin{align*}
				\nu_b^*(n)&\leq\frac{\log_{b}\log_{b}{n}}{\log_{b}{(\alpha^{-1})}} +\left( 1+ M+ \frac{\log_{b}{2}}{\log_{b}{(\alpha^{-1})}}\right)\\
				&=\frac{\log\log_{b}{n}}{\log{(\alpha^{-1})}} +\left( 1+ M+ \frac{\log_{b}{2}}{\log_{b}{(\alpha^{-1})}}\right).
			\end{align*}
			Dividing both sides of this inequality by $\log\log{n}$ and taking the $\limsup$ as $n\to \infty$, we get \eqref{upperbound} as desired. 
			
			\item[(ii)] \emph{Lower bound}. Here we shall use one of the ideas used in the proof of Theorem \ref{thm:erdosbase3}. 
			For each natural number $t$, let us consider $n_t=\left(\prod_{p_i \text{ prime}, p_i<b, p_i \nmid b} p_i\right)^{m_t}$, where $m_t$ is \emph{the smallest} positive integer such that 
			\begin{equation*}
			m_t\left(\frac{1}{2b}\right)^{t-1}\left(\log_b{2}\right)^{t-1} \geq  C=\max\{b,N\},
			\end{equation*}
			and where $N$ is given by Conjecture \ref{conj:uniformdistr} taking $\varepsilon=\dfrac{1}{2b}$. As we saw in the proof of Theorem \ref{thm:erdosbase3}, we have $\nu_b^*(n_t)\geq t$. Now, by the very definition of $m_t$, we know that 
			\begin{equation*}
			(m_t-1)\left(\frac{1}{2b}\right)^{t-1}\left(\log_b{2}\right)^{t-1}  <  C.
			\end{equation*}
			Taking logarithms (to base $b$) on both sides of this inequality and solving for $t$, we get
			\begin{equation*}
			t> \frac{\log_{b}{(m_t-1)}}{\log_{b}{\left(\beta^{-1}\right)}} +1 -
			\frac{\log_{b}{C}}{\log_{b}{\left(\beta^{-1}\right)}},
			\end{equation*}
			where $\beta=\frac{\log_{b}{2}}{2b}$. Note that 
			$
			\log_{b}{(m_t-1)} > \log_{b}{m_t} -\log_{b}{2}
			$
			(because $m_t>2$). Moreover, again from the definition of $m_t$, we have 
			\begin{equation*}
			\log_{b}{m_t} =\log_{b}\log_{b}{n_t} - \log_{b}\log_{b}\left(\prod_{1\leq i\leq k} p_i\right).
			\end{equation*}
			Putting all these facts together, we deduce that 
			\begin{align*}
			\nu_b^*(n_t) &> \frac{\log_{b}\log_{b}{n_t}}{\log_{b}{\left(\beta^{-1}\right)}} + K\\
			&= \frac{\log\log_{b}{n_t}}{\log{\left(\beta^{-1}\right)}} + K \stepcounter{equation}\tag{\theequation}\label{eq:nuestimate},
			\end{align*}
			where $K$ is a constant, namely
			\begin{equation*}
			K = 1- \frac{1}{\log_{b}\left(\beta^{-1}\right)}\log_{b}\left[2C\log_{b}{\left(\prod_{1\leq i\leq k}p_i\right)}\right].
			\end{equation*}
			Thus, the inequalities in (\ref{eq:nuestimate}) divided by $\log\log n_t$, letting $t\to\infty$, imply that 
			\begin{equation*}
			\limsup_{t\to\infty} \frac{\nu_b^*(n_t)}{\log\log{n_t}}\geq \frac{1}{\log{(\beta^{-1})}},
			\end{equation*}
			and this obviously implies \eqref{lowerbound}. 
		\end{enumerate}
	\end{proof}
	
	\subsection{$1$-shifted problem}
	
	In this section, we generalize a result of Wagstaff \cite{wagstaff1981iterating} for base 10, showing that for any base $b$, every positive integer $n$ reaches reach a fixed point or a cycle under iteration by $S_{1,b}$.	
	\begin{theorem}\label{thm:f1bconverges}
		Let $b \geq 2$. Then, for every positive integer $n$, the iterates of $S_{1,b}$
		starting at $n$ reach one of the following cycles:
		
		\begin{align*}
		&(10_2), & \text{ if }  b = 2;\\
		&(2,\dots,b-1,10_b) \text{ or } (1(b-2)_b),& \text{ if }  b \geq 3.
		\end{align*}
		
	\end{theorem}	
	\begin{proof}
		For ease of notation, let $f$ denote $S_{1,b}$ in this proof. First of all, notice that the theorem is trivial for $b = 2$, since in this case $f(n)$ is a power of $2$ for every $n$, and then $f^2(n) = 10_2$, which is the only fixed point of $f$. From this point on, we assume that $b \geq 3$.
		
		Assume first that $n$ is of the form $db^k-1$, where $2 \leq d \leq b$ (i.e., all the digits of $n$, possibly with the exception of the leading digit, are equal to $b-1$; the leading digit of $n$ is $d-1$; and $n$ has exactly $k$ digits). In this case, $f(n) = db^{k-1} = n+1$, and $2 \leq f(f(n)) \leq b = 10_b$, so this number belongs to the cycle $(2,3,\dots,10_b)$. 
		
		If $n$ is not of the form above, let $n = (d_kd_{k-1}\dots d_{0})_b$ be the representation of $n$ in base $b$, and let $j$ be the biggest index $i$ such that $i < k $ and $d_i < b-1$. We can bound $f(n)$ from above by $(d_k+1)(d_j+1)b^{k-1}$, which can be written as $d_kb^k+d_jb^{k-1}+b^{k-1}(1+d_k(d_j-(b-1)))$.
		
		On the other hand, we have that $n = \sum_{i=0}^k d_ib^i \geq d_kb^k + d_jb^{k+1}$. If the term $b^{k-1}(1+d_k(d_j-(b-1)))$ is negative, we have $f(n) < n$. As $d_j \leq b-2$, this term is nonnegative only if $d_j = b-2$ and $d_k = 1$. In this case, the bound for $f(n)$ becomes equal to $db^k+(b-2)b^{k-1}$. Furthermore, if $j < k-1$, then the lower bound on $n$ can be strengthened to $n \geq b^k + (b-1)b^{k+1} + (b-2)b^j$ and hence $f(n) < n$. So we must have $j = k-1$ to have $f(n) \geq n$. Also, if any digit of $n$ other than $d_k$ and $d_{k-1}$ is not zero, we have $n > db^k+(b-2)b^{k-1} \geq f(n)$. Therefore, $f(n) < n$ unless $n$ is of the form $1(b-2)000\dots0_b = b^k+(b-2)b^{k-1}$. In this case, $f(n) = 2(b - 1)$, which implies that $f(n) < n$ unless $k = 1$, which corresponds to the fixed point $n = 1(b-2)_b$. This proves that $f$ either reaches the fixed point $1(b-2)_b$ or keeps decreasing until it enters the cycle $(2,3,\dots,10_2)$.
		
	\end{proof}	
	In the case $b=3$, we are able to tell which integers reach the fixed point and which integers reach the cycle. Namely, we have the following result.	
	\begin{theorem}\label{t1b3cycles}
		For every $n \geq 1$, the sequence $(S^k_{1,3}(n))_{k \geq 1}$ reaches the cycle $(2,10_3)$ if and only if either $n$ or $2^{\#1(n)_3}$ lacks the digit 1 in its ternary expansion; otherwise, it reaches $(11_3)$.
	\end{theorem}	
	\begin{proof}
		Once more, let $f$ denote $S_{1,3}$. The result is clear for $n \in \{1,2,10_3,11_3\}$. Let $n \geq 12_3$. Notice that $f(n) = 2^{\#1(n)}3^{\#2(n)}$, so $f(n)$ ends in $\#2(n)$ zeros in base 3, and hence $f(f(n)) = f(2^{\#1(n)})$.
		
		If $\#1(n) = 0$, then $f(f(n)) = f(1) = 2$. Also, if $\#1(2^{\#1(n)}) = 0$, then $f^3(n) = f(f(2^{\#1(n)})) = f(2^{\#1(2^{\#1(n)})}) = f(1) = 2$.
		
		On the other hand, assume that both $\#1(n)$ and $\#1(2^{\#1(n)})$ are positive. We will prove by induction on $n$ (over those values such that $\#1(n) > 0$ and $2^{\#1(n)} > 0$) that, in this case, $(f^k(n))_{k \geq 1}$ reaches the cycle $(11_3)$.
		
		The result if trivial if $n \leq 11_3$. If $n > 11_3$, then $f(f(n)) = f(2^{\#1(n))})$. As $2^{\#1(n)} < n$, we can use the induction hypothesis to prove that $n$ reaches $(11_3)$ if we have $\#1(2^{\#1(n)}) > 0$ and $\#1(2^{\#1(2^{\#1(n)})}) > 0$. The first inequality is just part of the condition on $n$; the second comes from the fact that $\#1(2^{\#1(n)})$ is even (an even number must have an even number of digits 1 in its ternary expansion), and hence $2^{\#1(2^{\#1(n)})}$ is a power of four, and so it ends with a digit 1.
		
	\end{proof}	
	\begin{remark}
		By Conjecture \ref{conj:equidist}, the number of $n$ such that $\#1(2^n)_3  = 0$ is finite. A result of Narkiewicz \cite{narkiewicz1980note} says that the number of $n$ up to $N$ with this property is bounded by $1.62N^{\log_3 2}$, so in particular their density in the set of positive integers is zero (this fact also follows from a more general result of de Faria and Tresser \cite[Corollary~3.7]{de2014sloane}).
	\end{remark}	
	As for the persistence, similarly as in Theorem \ref{thm:erdosbase3}, we prove that Conjectures  \ref{conj:equidist} and \ref{conj:uniformdistr} imply that there are integers of arbitrarily large persistence for $S_{1,b}$.	
	\begin{theorem}\label{thm:t1b3pers}
		Conjecture \ref{conj:equidist} implies that there are integers of arbitrarily large persistence for the $1$-shifted problem in base $3$ and $4$, and Conjecture \ref{conj:uniformdistr} implies the same result for bases greater than $4$.
	\end{theorem}
	\begin{proof}
		We split the proof into three cases:
		\begin{enumerate}[(i)]
			
			\item $\textit{Base 3}$
			
			Put $f=S_{1,3}$. Conjecture \ref{conj:equidist} implies the following: for every $\varepsilon > 0$, there exists $N$ such that $2^m$ is $\varepsilon$-equidistributed whenever $m \geq N$. 
			To construct a number of persistence greater than $t$, notice that, if a number $a$ has exactly $x$ digits $1$ in its ternary expansion, then $f^2(a)=f(2^x)$. Take $\varepsilon = 1/6$ and let $N$ be the integer given by Conjecture \ref{conj:equidist} for this value of $\varepsilon$. Let $m$ be such that $m(1/6)^{t-1}(\log_3 2)^{t-1} \geq \max\{N,3\}$.
			
			We claim that the number $n=2^m$ has persistence larger than $t$. Let us denote, for $i \geq 1$ by $a_i$ and $b_i$, respectively, the numbers defined inductively in the following way: $2^m$ has $a_1$ digits $1$ and $b_1$ digits $2$ in its ternary expansion, and $f(2^{a_i})=2^{a_{i+1}}\cdot 3^{b_{i+1}}$ for every $i \geq 1$. As $m \geq N$, $2^m$ is $\varepsilon$-equidistributed, so we have $a_1 \geq (1/6)(\log_3 2)m$.
			This implies that $f^2(2^m)=f(2^{a_1}\cdot 3^{b_1})=f(2^{a_1})=2^{a_2}\cdot 3^{b_2}$. 
			As $a_1 \geq (1/6)\log_3 2 \cdot m \geq N$, the number $2^{a_1}$ is $\varepsilon$-equidistributed, and this in turn implies that $a_2 \geq (1/6)\log_3(2) \cdot a_1\geq (1/6)^2(\log_3{2})^2m$.
			Inductively, we have that $f^{t-1}(2^m) = 2^{a_{t-1}}\cdot 3^{b_{t-1}}$ where $a_{t-1} \geq (1/6)^{t-1}(\log_3 2)^{t-1}m \geq 3$, which implies that $n=2^m$ has persistence at least $t$, since Theorem \ref{thm:f1bconverges} shows that the elements of the cycle and the fixed point of $f$ have at most two digits.
			
			\item $\textit{Base 4}$

			We consider the number $n=3^m$, where $m(1/8)^{t-1}(\log_4{3})^{t-1} \geq \max\{M,4\}$, where $M$ is the maximum of the $N$ obtained applying Conjecture \ref{conj:equidist} with $(q,F,a)=(4,\{3\},1)$ and $(q,F,a)=(4,\{3\},2)$, in both cases with $\varepsilon = 1/8$, and the proof of item (i) applies.
			
			\item $\textit{Larger bases}$
			
			Finally, for $b \geq 5$, let $F$ be the set of primes smaller than $b$ and $F_r=F-\{r\}$. We apply Conjecture \ref{conj:uniformdistr} for each prime divisor $r$ of $b$ and each proper divisor $d$ of $b$, with $F_r$ being the set of primes considered, $q=b$ and $a=d$. Note that, by Bertrand's postulate (which states that, for every $n>1$, there is a prime $p$ such that $n<p<2n$), there is some prime $s$ such that $(b-1)/2<s<b$. In particular, $s$ does not divide $b$ (so that $F_r$ is nonempty for every prime $r$ dividing $b$). Taking the maximum of the $N$ given in each application of the conjecture with a fixed $\varepsilon > 0$, we get the following statement: for every $\varepsilon > 0$, there is $N$ such that $d\prod_{p_i \in F_r}p_i^{\alpha_i}$ is $\varepsilon$-equidistributed whenever $\alpha_i \geq N$ for some $i \in \{1,\dots,k\}$, where $d$ is a proper divisor of $b$ and $r$ is a prime divisor of $b$.
			Moreover, note that, for every $n$, one has $f(bn)=f(n)$, since the expansion of $bn$ and $n$ in base $b$ differ only by one digit $0$.
			
			We claim that $\nu_{1,b}(n) \geq t$, where $n=(\prod_{p_i \in F, p_i\nmid b}p_i)^m$ and $m$ satisfies the inequality $m(1/(2b))^{t-1}(\log_b 2)^{t-1} \geq \max\{N,b\}$, where $N$ is the integer given by the applications of Conjecture \ref{conj:uniformdistr} as above with $\varepsilon = 1/(2b)$.
			As $m \geq N$, $n$ is $\varepsilon$-equidistributed, whence, using a rough bound $(\prod_{p_i \in F, p_i\nmid b}p_i)^m \geq 2^m$, we have $f(n)=\prod_{p_i \leq b \text{ prime }}p_i^{\beta_i}$, with $\beta_i \geq (1/b-\varepsilon)\#(n)_b \geq (1/(2b))(\log_b 2)m$. The fact that $f(ba)=f(a)$ for every $a$ implies that $f(f(n))=f(d\prod_{p_i \in F_ r}p_i^{\alpha_i})$ for some proper divisor $d$ of $b$ and some prime $r$ dividing $b$. As $\alpha_i = \beta_i$ for every $p_i$ that does not divide $b$ (because no power of $p_i$ can be factored out into powers or divisors of $b$), we have $\alpha_i \geq (1/(2b))(\log_b 2)m \geq N$ for some $i \in \{1,\dots,k\}$. Then, the number $d\prod_{1\leq i \leq k}p_i^{\alpha_i}$ is $\varepsilon$-equidistributed, and hence we have $f(f(n))=\prod_{p_i \leq b \text{ prime }}p_i^{\beta'_i}$ with $\beta'_i \geq (1/(2b))\#(d\prod_{1\leq i \leq k}p_i^{\alpha_i})_b \geq(1/(2b))(\log_b 2)\alpha_i \geq (1/(2b))^2(\log_b 2)^2 \cdot m$ for every $i$, and the argument can be repeated. Inductively, the exponents of the $p_i$ in $f^{t-1}(n)$ are at least $(1/(2b))^{t-1}(\log_b 2)^{t-1} \cdot m$. In particular, $f^{t-1}(n) \geq 2^{(1/(2b))^{t-1}(\log_b 2)^{t-1} \cdot m} \geq 2^{b} \geq b^2$. Again, as the elements of the cycle and the fixed point of $f$ have at most two digits, this implies that $n$ has persistence at least $t$.
		\end{enumerate}	
	\end{proof}	
	An alternative proof of Theorem \ref{thm:t1b3pers} in the case $b=3$ would be to find an infinite sequence $(a_n)_{n \geq 0}$ such that $2^{a_n}$ has $a_{n-1}$ digits $1$ in base $3$ for every $n \geq 1$. In this case, the integer $2^{a_n}$ would have persistence equal to $n$ plus the persistence of $2^{a_0}$. The existence of such a sequence is a straightforward consequence of Conjecture \ref{conj:equidist}, but we conjecture it independently, as it may be much simpler to prove than the full statement of the original conjecture. A computational search shows that the initial terms of such a sequence could be 
	\begin{equation*}
	2, 4, 8, 24, 96, 350, 1580, 7520, 35600, 168980,
	\end{equation*}
	since $2^{168980}$ has $35600$ digits $1$, $2^{35600}$ has $7520$ digits $1$, and so on, and $2^2$ is a fixed point of $S_{1,3}$. In fact, there is some computational evidence in favor of the following stronger conjecture, which assures that one can find such a sequence starting from any sufficiently large even number:	
	\begin{conjecture}\label{conj:dig1pow2}
		There is $N$ such that, for every $n > N$, there is $m$ such that $2^{2m}$ has exactly $2n$ digits $1$ in base $3$.
	\end{conjecture}	
	\begin{remark}
		Conjecture \ref{conj:dig1pow2} is equivalent to the assertion that the subsequence of terms of even order of \seqnum{A036461} contains all sufficiently large even numbers.
	\end{remark}	
	Just as we did for the Erd\H{o}s-Sloane persistence, we can estimate the maximal growth of $\nu_{1,b}(n)$ as a function of $n$ from above (unconditionally) and from below (using Conjecture \ref{conj:uniformdistr} as in Theorem \ref{thm:t1b3pers}). More precisely, we have the result stated below. In its proof, we shall use the following evident fact. If $F$ is any finite non-empty set and $\phi:F\to F$ is a self-map, then every $x\in F$ is eventually mapped to a (fixed or) periodic point under $\phi$, and the number of iterates that it takes for $x$ to reach the periodic cycle is obviously bounded by the cardinality of $F$.
	\begin{theorem}\label{thm:perst1}
		For each base $b\geq 3$, we have 
		\begin{equation}\label{shiftupperbound}
		\limsup_{n\to\infty} \frac{\nu_{1,b}(n)}{\log\log{n}}\leq \frac{2}{\log\left(\alpha^{-1}\right)},
		\end{equation}
		where $\alpha=\log_{b}{(b-1)}$. Moreover, if Conjecture \ref{conj:uniformdistr} holds, then we have
		\begin{equation}\label{shiftlowerbound}
		\limsup_{n\to\infty} \frac{\nu_{1,b}(n)}{\log\log{n}}\geq \frac{1}{\log{(\beta^{-1})}},
		\end{equation}
		where $\beta=\frac{\log_{b}{2}}{2b}$.
	\end{theorem}	
	\begin{proof}
		Let us write $f=S_{1,b}$ in this proof. We treat the upper and lower estimates separately.
		
		\begin{enumerate}
			\item[(i)] \emph{Upper bound}. Given $n>b$, write $n=(d_1d_2\cdots d_k)_b$ (where $k\leq 1+\log_{b}{n}$; note that this notation differs from the one in section \ref{sec:def}) and let $\ell\leq k$ be the number of digits $d_i$ that are equal to $b-1$. Then $f(n)=b^\ell\cdot P$, where $P=\prod_{d_i\leq b-2} (1+d_i)$. This in turn implies that $f^2(n)=f(P)$. But the number of digits of $P$ in base $b$ is at most 
			\begin{equation*}
			1+\log_b{P} = 1+ \sum_{d_i\leq b-2}\log_b{(1+d_i)} \leq 1+(k-\ell)\alpha\leq 1+k\alpha,
			\end{equation*}
			where $\alpha=\log_b{(b-1)}<1$. Hence we have
			\begin{equation*}
			f^2(n)=f(P)\leq b^{1+k\alpha}\leq b^{1+(1+\log_b{n})\alpha}=b^{1+\alpha}n^\alpha.
			\end{equation*}
			By induction, it follows that
			\begin{equation*}
			f^{2j}\leq \left(b^{1+\alpha}\right)^{1+\alpha+\alpha^2+\cdots+\alpha^{j-1}}n^{\alpha^{j}}
			< b^{(1+\alpha)/(1-\alpha)}n^{\alpha^{j}}, \forall j\geq 1.
			\end{equation*}
			Let $j_0$ be the smallest natural number such that $n^{\alpha^{j_0}}<2$, i.e.,
			\begin{equation}\label{jay0}
			j_0= \left\lceil \frac{\log\log{n} - \log\log{2}}{\log{(\alpha^{-1})}}\right\rceil.
			\end{equation}
			Then we have $f^{2j_0}(n)\in A =\{1,2,\ldots,M\}$, where $M=\left\lceil 2b^{(1+\alpha)/(1-\alpha)}\right\rceil$. We claim that $A$ is $f^2$-invariant, i.e., $f^2(A)\subseteq A$. Indeed, if $a\in A$ then 
			\begin{equation*}
			f^2(a)\leq b^{1+\alpha}a^{\alpha} \leq b^{1+\alpha} M^{\alpha} \leq b^{1+\alpha} \left(2b^{(1+\alpha)/(1-\alpha)}\right)^{\alpha} = 2^\alpha b^{(1+\alpha)/(1-\alpha)} < M,
			\end{equation*}
			and so $f^2(a)\in A$ as claimed. But now, by the simple remark preceding the statement of this theorem, every $a\in A$ is eventually periodic, and if $m\in \mathbb{N}$ is the smallest number such that $f^{2m}(a)\in \mathrm{Per}(f^2)\subseteq \mathrm{Per}(f)$, then $m\leq |A| = M$. Summarizing, we have proved that, starting from $n>b$: (i) after $2j_0$ iterates under $f$, we reach some $a_0\in A$; (ii) after $j_1\leq 2M$ further iterates, we reach a periodic cycle, i.e., $f^{j_1}(a_0)\in \mathrm{Per}(f)$. Therefore we have $\nu_{1,b}(n)\leq 2j_0 + 2M$, and from \eqref{jay0} we deduce that 
			\begin{equation*}
			\nu_{1,b}(n)\leq 2\left( \frac{\log\log{n} - \log\log{2}}{\log{(\alpha^{-1})}} \right) +2(M+1).
			\end{equation*}
			This shows that 
			\begin{equation*}
			\limsup_{n\to\infty} \frac{\nu_{1,b}(n)}{\log\log{n}}\leq \frac{2}{\log\left(\alpha^{-1}\right)}.
			\end{equation*}
			and this is precisely \eqref{shiftupperbound}. 
			
			\item[(ii)] \emph{Lower bound}. Here we proceed just as in the proof of the lower bound in Theorem \ref{thm:t1b3pers}. Once again, for each natural number $t$ we consider $n_t=\left(\prod_{p_i \text{ prime}, p_i<b, p_i \nmid b} p_i\right)^{m_t}$, where $m_t$ is \emph{the smallest} positive integer such that 
			\begin{equation*}
			m_t\left(\frac{1}{2b}\right)^{t-1}\left(\log_b{2}\right)^{t-1} \geq C=\max\{b,N\},
			\end{equation*}
			and where $N$ is given by Conjecture \ref{conj:uniformdistr} taking $\varepsilon=\dfrac{1}{2b}$. Then, as we saw in the proof of Theorem \ref{thm:t1b3pers}, the $1$-shifted persistence of $n_t$ is at least $t$, and the same computations performed in the proof of Theorem \ref{thm:perserdos} yield
			\begin{equation}\label{shiftnuestimate}
			\nu_{1,b}(n_t) \geq t > \frac{\log_{b}\log_{b}{n_t}}{\log_{b}{\left(\beta^{-1}\right)}} + K=\frac{\log\log_{b}{n_t}}{\log{\left(\beta^{-1}\right)}} + K,
			\end{equation}
			for some constant $K$. Dividing the resulting inequality in \eqref{shiftnuestimate} by $\log\log{n}_t$ and letting $t\to\infty$, we arrive at \eqref{shiftlowerbound} as desired.  
		\end{enumerate}
	\end{proof}

	\subsection{$2$-shifted problem}
	
	In this section, we show that Conjectures \ref{conj:equidist} and \ref{conj:uniformdistr} imply that every integer reaches a cycle or a fixed point under iteration by $S_{2,b}$. Before stating the result precisely, we start with a lemma.	
	\begin{lemma}\label{lem:forthmt2blarge}
		Let $b\geq 5$ be a positive integer. Then
		\begin{enumerate}
			\item $(b+1)!^{\frac{\log_b{(b+1)}}{b}}<b$;
			
			\item $(b+1)^{\log_b(b-1)}<b$;
			
			\item $(b+1)^{2\log_b(b-2)+\log_b(b+1)}<b^3$.
		\end{enumerate}  
	\end{lemma}	
	\begin{proof}
		Taking logarithms and rearranging terms, the first inequality is equivalent to
		\begin{equation}\label{eq:lem15}
		b(\log{b})^2 - \log(b+1)!\cdot \log(b+1) > 0.
		\end{equation}
		We will use the following well-known upper bound for $n!$, valid for all positive integers $n$:
		
		\begin{equation*}
		n! \leq e\left(\frac{n}{e}\right)^n\sqrt{n}.
		\end{equation*}
		Applying this bound to the left-hand side of inequality (\ref{eq:lem15}), we get
		\begin{align*}
		b(\log{b})^2 - \log(b+1)!&\cdot \log(b+1)  \\ &\geq b(\log{b})^2-(b+1)(\log(b+1))^2+b\log(b+1)-\frac{(\log(b+1))^2}{2}.
		\end{align*}
		By the mean value theorem, $b(\log{b})^2-(b+1)(\log(b+1))^2 = -g'(c)$ for some $c \in (b,b+1)$, where $g(x)=x(\log{x})^2$. As $g'(x)=(\log{x})^2+2\log{x}$ is increasing, we get that $b(\log{b})^2-(b+1)(\log(b+1))^2 \geq -(\log(b+1))^2-2\log(b+1)$. This implies that
		
		\begin{align*}
		b(\log{b})^2 - \log(b+1)!\cdot \log(b+1) & \geq b\log(b+1)-\frac{3(\log(b+1))^2}{2}-2\log(b+1).\\
		\end{align*}
		
		Let $h(b) = b\log(b+1)-\frac{3(\log(b+1))^2}{2}-2(\log(b+1))$. It is readily checked that $h(5)> 0$. Moreover, $h'(b)=\frac{(b - 2) (\log(b + 1) + 1)}{b + 1} > 0$ for every $b > 2$. This implies that $h(b) > 0$ for all $b \geq 5$ and concludes the proof of the first item.
		
		The inequality of the second item is equivalent, taking logarithms and rearranging terms, to
		\begin{equation}\label{eq:log}
			\frac{\log(b-1)}{\log{b}}<\frac{\log{b}}{\log(b+1)}.
		\end{equation}
		Inequality (\ref{eq:log}) is a straightforward consequence of the fact that $f(x)=\frac{\log{x}}{\log(x+1)}$ is increasing for $x>1$, which follows immediately from $f'(x)=\frac{(x+1)\log(x+1)-x\log{x}}{x(x+1)(\log(x+1))^2}$.
		
		Finally, the third inequality is equivalent to 
		\begin{equation}\label{eq:log2}
			\log(b+1)(2\log(b-2)+\log(b+1))<3(\log{b})^2.
		\end{equation}
		This can be proved using that $\log$ is a concave function and applying Jensen's inequality
		to $2\log(b-2)+\log(b+1)$:
		\begin{equation*}
			2\log(b-2)+\log(b+1)<3\log(b-1).
		\end{equation*}
		The left-hand side of (\ref{eq:log2}) is, then, smaller than $3\log(b-1)\log(b+1)$. Inequality (\ref{eq:log}) implies that this is less than $3(\log{b})^2$ and concludes the proof.
	\end{proof}	
	\begin{theorem}\label{thm:f23converges}
		Conjecture \ref{conj:equidist} (resp., Conjecture \ref{conj:uniformdistr}) implies that for every $n \geq 1$, the iterates of $S_{2,3}(n)$ (resp., $S_{2,b}(n)$, for $b \geq 4$) reach a cycle or a fixed point.
	\end{theorem}	
	\begin{proof}
		We will prove the result for $b=3$, $b=4$ and $b \geq 5$ separately. In any case, to show that the sequence of iterates starting at any positive integer stabilizes, we will prove the following stronger statement: there exist integers $N_0$ and $k$ and a constant $0 < c_b < 1$ (which depends only on $b$) such that, for every $n \geq N_0$, $S_{2,b}^j(n) \leq n^{c_b}$ (or, equivalently, $S_{2,b}^j(n) \leq C\cdot n^{c'_b}$ for some constant $C$ independent of $n$ and $0 < c'_b < 1$) for some $1 \leq j \leq k$. In the cases $b=3$, $b=4$ and $b \geq 5$, 
		we will prove this statement with $k=4$, $k=3$ and $k=2$, respectively.

		\begin{enumerate}[(i)]
			
			\item  \textit{Base 3}
			
			First, put $f=S_{2,3}$ to simplify the notation. We will use the following instance of Conjecture \ref{conj:equidist}, which corresponds to $q=3$, $F=\{2\}$ and $a=1$: for every $\varepsilon > 0$, there exists $N$ such that $2^t$ is $\varepsilon$-equidistributed in base $3$ whenever $t \geq N$.
			
			Let $a_0, b_0, c_0$ denote, respectively, $\#0(n), \#1(n), \#2(n)$; and, for $k \geq 1$, let $a_k, b_k, c_k$ denote,  respectively, $\#0(2^{b_{k-2}+a_{k-1}+2c_{k-1}}), \#1(2^{b_{k-2}+a_{k-1}+2c_{k-1}}),
			\#2(2^{b_{k-2}+a_{k-1}+2c_{k-1}})$, where we put $b_{-1}=0$. With this notation, we have
			\begin{equation*}
			f^k(n)=2^{b_{k-2}+a_{k-1}+2c_{k-1}}\cdot 3^{b_{k-1}} 
			\end{equation*} 
			for every $k \geq 1$.

			Let us write $\alpha = \log_3 2 \approx 0.631$. Moreover, fix $\varepsilon = 0.001$ and write $\delta = 1/3-\varepsilon$. Let $N$ be such that $2^t$ is $\varepsilon$-equidistributed in base $3$ for every $t \geq N$.
			
			For any integer $n$, we have $f(n) = 2^{a_0+2c_0}\cdot 3^{b_0}$ and $f^2(n)=2^{b_0}f(2^{a_0+2c_0})$.
			
			Let $n \geq 3^{3M}$ be an integer, where $M \geq N/(\delta\alpha)^2$. We know that $a_0+b_0+c_0 \geq \log_3 n \geq 3M$. This implies that either $b_0 \geq 2M$ or $a_0+c_0 \geq M$.
			
			Suppose first that $a_0+c_0 < M$. Then $b_0 \geq 2M$, and using the trivial bound $f(m) \leq 4^{\log_3 m + 1}$, which holds for every $m$ since each of the at most $\log_3 m + 1$ digits of $m$ is mapped to a factor $2$, $3$ or $4$ in $f(m)$, we get, using that $3^{a_0+b_0+c_0-1} \leq n \leq 3^{a_0+b_0+c_0}$, that
			
			\begin{align*}
			\frac{f^2(n)}{n}&\leq \frac{2^{b_0}f(2^{a_0+2c_0})}{3^{a_0+b_0+c_0-1}}\\
			& \leq \frac{2^{b_0}\cdot 4^{\log_3(2^{a_0+2c_0})+1}}{3^{a_0+b_0+c_0-1}}\\
			&= c\cdot 3^{(\alpha-1) b_0+(2\alpha^2-1)a_0+(4\alpha^2-1)c_0}\\
			& \leq  c\cdot \left(3^{b_0}\right)^{\alpha-1+(4\alpha^2-1)/2}\\
			&\leq c\cdot \left(\frac{n}{3^M}\right)^{\alpha-1+(4\alpha^2-1)/2},
			\end{align*}
			for some positive constant $c$, since $1-2\alpha^2$, $1-\alpha > 0$ and $1-4\alpha^2 < 0$. As $\alpha-1+(4\alpha^2-1)/2<0$, this concludes the proof in this case.
			
			From now on, we may assume that $a_0+c_0 \geq M$. As $M \geq N/(\delta\alpha)^2 \geq N$, we know that $2^{a_0+2c_0}$ is $\varepsilon$-equidistributed, i.e., $a_1,b_1,c_1$ belong to $((1/3-\varepsilon)\alpha(a_0+2c_0),(1/3+\varepsilon)\alpha(a_0+2c_0))$. In particular, writing $\beta$ for $1/3+\varepsilon$, we have that $a_1,b_1,c_1$ are bounded from above by 
			\begin{equation*}
				\beta\alpha(a_0+2c_0).
			\end{equation*}			
			As $f^2(n)=2^{b_0+a_1+2c_1}\cdot 3^{b_1}$, we have $f^3(n) = 2^{b_1}f(2^{b_0+a_1+2c_1})$. As $a_1 \geq \delta\alpha(a_0+2c_0)\geq\delta\alpha M \geq N$, the number $2^{b_0+a_1+2c_1}$ is $\varepsilon$-equidistributed, i.e., the quantities $a_2, b_2, c_2$ belong to the interval $((1/3-\varepsilon)\alpha(b_0+a_1+2c_1),(1/3+\varepsilon)\alpha(b_0+a_1+2c_1))$, and hence are bounded from above by
			\begin{align*}
				\beta\alpha(b_0+a_1+2c_1)&\leq \beta\alpha(b_0+3\beta\alpha(a_0+2c_0))\\
				&=3\beta^2\alpha^2(a_0+2c_0)+\beta\alpha b_0.
			\end{align*}
			On the other hand, the numbers $a_2, b_2, c_2$ are greater than $\delta\alpha(b_0+a_1+2c_1)>\delta^2\alpha^2(a_0+2c_0) \geq N$, so $2^{a_2+2c_2+b_1}$ is $\varepsilon$-equidistributed. This implies that 
			\begin{align*}
			a_3,b_3,c_3 &\leq \beta\alpha(b_1+a_2+2c_2)\\
			&\leq\beta\alpha(\beta\alpha(a_0+2c_0)+9\beta^2\alpha^2(a_0+2c_0)+3\beta\alpha b_0)\\
			&=\beta^2\alpha^2(1+9\beta\alpha)(a_0+2c_0)+3\beta^2\alpha^2 b_0.
		    \end{align*}
		    Together, these estimates imply that
		    \begin{align*}
		    	\alpha(b_2+a_3+2c_3)+b_3&\leq\alpha(3\beta^2\alpha^2(a_0+2c_0)+\beta\alpha b_0)+\\
		    	&\qquad(3\alpha+1)(\beta^2\alpha^2(1+9\beta\alpha)(a_0+2c_0)+3\beta^2\alpha^2 b_0)\\
		    	&=\beta^2\alpha^2(3\alpha+(3\alpha+1)(1+9\beta\alpha))(a_0+2c_0)+\beta\alpha^2(1+3(3\alpha+1)\beta)b_0.
		    \end{align*}
			Finally, this bound implies that
			\begin{align*}
			\frac{f^4(n)}{n} & = \frac{2^{b_2+a_3+2c_3}\cdot 3^{b_3}}{3^{a_0+b_0+c_0-1}}\\
			&\leq 3\cdot\frac{3^{\alpha(b_2+a_3+2c_3)+b_3}}{3^{a_0+b_0+c_0}}\\
			&\leq 3\cdot\frac{3^{\beta^2\alpha^2(3\alpha+(3\alpha+1)(1+9\beta\alpha))(a_0+2c_0)+\beta\alpha^2(1+3(3\alpha+1)\beta)b_0}}{3^{a_0+b_0+c_0}}\\
			&= 3\cdot{3^{(\beta^2\alpha^2(3\alpha+(3\alpha+1)(1+9\beta\alpha))-1)a_0+(\beta\alpha^2(1+3(3\alpha+1)\beta)-1)b_0+(2\beta^2\alpha^2(3\alpha+(3\alpha+1)(1+9\beta\alpha))-1)c_0}}.
			\end{align*}
			This gives the result, since
			$\beta^2\alpha^2(3\alpha+(3\alpha+1)(1+9\beta\alpha)) < \frac{1}{2}$ and $\beta\alpha^2(1+3(3\alpha+1)\beta)<1$. 
			
			\item \textit{Base 4}
			
			We now consider base $4$. As usual, put $f=S_{2,4}$ to ease the notation. 
			We will prove the following: for every sufficiently large integer $n$, one of the numbers $f(n)$, $f^2(n)$, $f^3(n)$ is at most $cn^{c_b}$ for some $c>0$ and $0<c_b<1$. We will apply Conjecture \ref{conj:uniformdistr} twice, with $q=4$, $F=\{3,5\}$, $a=1$ and $q=4$, $F=\{3,5\}$, $a=2$ to get the following statement: for every $\varepsilon > 0$, there is $N$ such that $3^x\cdot 5^y$ and $2\cdot 3^x \cdot 5^y$ are $\varepsilon$-equidistributed whenever $x \geq N$ or $y \geq N$. 		
			
			For any integer $n$, if we put $a_0=\#0(n)$, $b_0=\#1(n)$, $c_0=\#2(n)$ and $d_0=\#3(n)$, we have $f(n) = 4^{\lfloor{a_0/2}\rfloor+c_0}\cdot 2^{{a'_0}}\cdot 3^{b_0}\cdot 5^{d_0}$ and $f^2(n)=2^{\lfloor{a_0/2}\rfloor+c_0}\cdot f(2^{{a_0}'}\cdot 3^{b_0}\cdot 5^{d_0})$, where ${x'}$ denotes the remainder of the integer $x$ modulo $2$.
			
			Let $\varepsilon = 0.001$ and let $n \geq 4^{4M}$ be an integer, where $M \geq N/(2(1/4-\varepsilon)\log_4{3})$. We know that $a_0+b_0+c_0+d_0 \geq \log_4 n \geq 4M$. This implies that either $b_0 \geq M$, or $d_0 \geq M$, or $a_0+c_0 \geq 2M$.
			
			Suppose first that $b_0 < M$ and $d_0 < M$. Then $a_0+c_0 \geq 2M$, and using the trivial bound $f(m) \leq 5^{\log_4 m + 1}$, which holds for every $m$ since each of the at most $\log_4 m + 1$ digits of $m$ is mapped to a factor $2$, $3$, $4$, or $5$ in $f(m)$, we get that
			
			\begin{align*}
			\frac{f^2(n)}{n}&\leq\frac{2^{\lfloor{a_0/2}\rfloor+c_0}\cdot f(2^{{a'_0}}\cdot 3^{b_0}\cdot 5^{d_0})}{4^{a_0+b_0+c_0+d_0-1}}\\
			& \leq 4\cdot\frac{2^{a_0/2+c_0}\cdot 5^{\log_4(2\cdot 3^{b_0}\cdot 5^{d_0})+1}}{4^{a_0+b_0+c_0+d_0}}\\
			&=c \cdot 4^{-3a_0/4-c_0/2+(\log_4{5}\log_4{3}-1)b_0+((\log_4{5})^2-1)d_0}\\
			&\leq c(M) \cdot (4^{a_0+c_0})^{-1/2}\\
			& \leq c(M) \cdot \left(\frac{n}{4^{2M}}\right)^{-1/2},
			\end{align*}
			where $c(M)$ is some constant dependent on $M$ only. This proves the claim in the first case.
			
			From now on, we assume that $b_0 \geq M$ or $d_0 \geq M$. As $M \geq N/(2(1/4-\varepsilon)\log_4{3})\geq N$, we know that $2^{{a'_0}}\cdot 3^{b_0}\cdot 5^{d_0}$ is $\varepsilon$-equidistributed, i.e., $\#d(2^{{a'_0}}\cdot 3^{b_0}\cdot 5^{d_0})$ belongs to the interval $((1/4-\varepsilon)\log_4(2^{{a'_0}}\cdot 3^{b_0}\cdot 5^{d_0}),(1/4+\varepsilon)\log_4(2^{{a'_0}}\cdot 3^{b_0}\cdot 5^{d_0}))$, for $d=0,1,2,3$. Put $a_1=\#0(2^{{a'_0}}\cdot 3^{b_0}\cdot 5^{d_0})$, $b_1=\#1(2^{{a'_0}}\cdot 3^{b_0}\cdot 5^{d_0})$, $c_1=\#2(2^{{a'_0}}\cdot 3^{b_0}\cdot 5^{d_0})$, and $d_1=\#3(2^{{a'_0}}\cdot 3^{b_0}\cdot 5^{d_0})$. In particular, $a_1$, $b_1$, $c_1$ and $d_1$ are bounded from above by $(1/4+\varepsilon)\log_4(2^{{a'_0}}\cdot 3^{b_0}\cdot 5^{d_0}) \leq (1/4+2\varepsilon)\log_4(3^{b_0}\cdot 5^{d_0})\leq \beta\alpha(b_0+d_0)$ for $M$ large enough (namely, for $(1/4+\varepsilon)\log_4(2\cdot 3^{M}\cdot 5^{M}) \leq (1/4+2\varepsilon)\log_4(3^{M}\cdot 5^{M})$ to hold), writing $\beta$ for $1/4+2\varepsilon$ and $\alpha$ for $\log_4{5}$.
			
			We have 
			\begin{equation*}f^2(n)=2^{\lfloor{a_0/2}\rfloor+c_0}\cdot f(2^{{a'_0}}\cdot 3^{b_0}\cdot 5^{d_0})
			=2^{\lfloor{a_0/2}\rfloor+c_0+a_1+2c_1}\cdot 3^{b_1} \cdot 5^{d_1},
			\end{equation*} 
			and then
			\begin{equation*}
			f^3(n)=f(2^{\lfloor{a_0/2}\rfloor+c_0+a_1+2c_1}\cdot 3^{b_1} \cdot 5^{d_1})=2^{\lfloor{(\lfloor{a_0/2}\rfloor+c_0+a_1+2c_1)/2}\rfloor}f(2^{(\lfloor{a_0/2}\rfloor+c_0+a_1+2c_1)'}\cdot 3^{b_1}\cdot 5^{d_1}).
			\end{equation*}
			Putting
			$a_2=\#0(2^{(\lfloor{a_0/2}\rfloor+c_0+a_1+2c_1)'}\cdot 3^{b_1}\cdot 5^{d_1})$, 
			$b_2=\#1(2^{(\lfloor{a_0/2}\rfloor+c_0+a_1+2c_1)'}\cdot 3^{b_1}\cdot 5^{d_1})$, 
			$c_2=\#2(2^{(\lfloor{a_0/2}\rfloor+c_0+a_1+2c_1)'}\cdot 3^{b_1}\cdot 5^{d_1})$, and 
			$d_2=\#3(2^{(\lfloor{a_0/2}\rfloor+c_0+a_1+2c_1)'}\cdot 3^{b_1}\cdot 5^{d_1})$, we have
			
			\begin{equation}\label{eq:est}
			f^3(n)\leq 2^{a_0/4+c_0/2+a_1/2+c_1+a_2+2c_2}\cdot 3^{b_2} \cdot 5^{d_2}.
			\end{equation}
			
			As $b_1 \geq (1/4-\varepsilon)\log_4(2^{{a'_0}}\cdot 3^{b_0}\cdot 5^{d_0}) \geq (1/4-\varepsilon)\log_4(3^{b_0}\cdot 5^{d_0}) \geq 2M(1/4-\varepsilon)\log_4{3}\geq N$, 
			the number $2^{(\lfloor{a_0/2}\rfloor+c_0+a_1+2c_1)'}\cdot 3^{b_1}\cdot 5^{d_1}$ is $\varepsilon$-equidistributed,
			i.e., $a_2$, $b_2$, $c_2$ and $d_2$ belong to  $((1/4-\varepsilon)\log_4(2^{(\lfloor{a_0/2}\rfloor+c_+a_1+2c_1)'}\cdot 3^{b_1}\cdot 5^{d_1}),(1/4+\varepsilon)\log_4(2^{(\lfloor{a_0/2}\rfloor+c_+a_1+2c_1)'}\cdot 3^{b_1}\cdot 5^{d_1}))$.
			
			This implies that $a_2$, $b_2$, $c_2$ and $d_2$ can be  bounded from above by the following expression: $(1/4+\varepsilon)\log_4(2^{(\lfloor{a_0/2}\rfloor+c_0+a_1+2c_1)'}\cdot 3^{b_1}\cdot 5^{d_1})\leq (1/4+2\varepsilon)\alpha(b_1+d_1)\leq 2\beta^2\alpha^2(b_0+d_0)$. Plugging these estimates for $a_1, \dots, d_1,a_2,\dots,d_2$ in (\ref{eq:est}), we get that
			
			\begin{align*}
			\frac{f^3(n)}{n} &\leq \frac{4^{a_0/8+c_0/4+(3/4\alpha\beta+3\alpha^2\beta^2+4\alpha^3\beta^2)(b_0+d_0)}}{4^{a_0+b_0+c_0+d_0-1}}\\
			&=4 \cdot 4^{-7a_0/8-3c_0/4+(3/4\alpha\beta+3\alpha^2\beta^2+4\alpha^3\beta^2-1)(b_0+d_0)}\\
			& \leq 4 \cdot n^{\theta},
			\end{align*}
			for some $0<\theta<1$ since $3\alpha\beta/4+3\alpha^2\beta^2+4\alpha^3\beta^2 < 1$.
			
			\item \textit{Larger bases}
			
			Finally, we prove the result for $b \geq 5$. Let $f$ denote $S_{2,b}$. We will prove that either $f(n) \leq c\cdot n^{c_b}$ or $f^2(n) \leq c \cdot n^{c_b}$ if $n$ is large enough, for some $c>0$ and $0<c_b<1$.
			
			We start applying Conjecture \ref{conj:uniformdistr} a few times: 
			for every prime $p$ dividing $b$ and for every proper divisor $d$ of $b$ (i.e., a divisor which is less than $b$, including $1$), we apply it for $q=b$, $a=d$, $F_p=\{r \text{ prime}:r \leq b+1, r \neq p\}$. 
			Taking the maximum of the $N$ obtained by each application of the conjecture, we get the following statement: 
			for every $\varepsilon > 0$, there is $N$ such that, for every proper divisor $d$ of $b$ and every prime divisor $p$ of $b$, the number $d\prod_{q_i \in F_p} q_i^{a_i}$ is $\varepsilon$-equidistributed if any of the $a_i$ is at least $N$.
			
			Let $\varepsilon > 0$ be such that $(b+1)!^{(\log_b(b+1))(1/b+\varepsilon)}<b$. Such an $\varepsilon$ exists by the first item of Lemma \ref{lem:forthmt2blarge} and by the fact that $\lim_{\varepsilon \to 0^+}(b+1)!^{\varepsilon\log_b(b+1)} = 1$. Moreover, let $n \geq b^{4M}$, with $M \geq N$, where $N$ is as in the paragraph above. For $i \in \{0,\dots,b-1\}$, let $n_i$ denote the number of digits $i$ in the base-$b$ expansion of $n$. We know that $\sum_{i=0}^{b-1}n_i \geq \log_b{n} \geq 4M$.
			
			By the definition of $f$, we have $f(n)=\prod_{i=0}^{b-1}(i+2)^{n_i}$. We may rewrite this number as $b^t\cdot d \cdot \prod_{q_i \in F_p}q_i^{\alpha_i}$, where $t \geq 0$ is an integer, $d$ is a proper divisor of $b$ and $p$ is a prime divisor of $b$. In particular, we have $f(f(n))=2^tf(d \cdot \prod_{q_i \in F_p}q_i^{\alpha_i})$. Note that we have $t \geq n_{b-2}$, since all the $n_{b-2}$ powers of $b$ are factored out to the term $b^t$.

			Suppose first that $n_{b-1} \geq M$ or $n_{b-3} \geq M$. Then, the number $d \cdot \prod_{q_i \in F_p}q_i^{\alpha_i}$ is $\varepsilon$-equidistributed, since no prime factor from $b+1$ or $b-1$ is factored out in $b^t$ or $d$ in the product $b^t\cdot d \cdot \prod_{q_i \in F_p}q_i^{\alpha_i}$, as $b$ is coprime with $b-1$ and $b+1$, and hence $\alpha_i\geq \max\{n_{b-3},n_{b-1}\}\geq M$ for some $i$. The number of occurrences of every digit from $0$ to $b-1$ in $f(n)$ belongs, then, to the interval $((1/b-\varepsilon)\log_b(d \cdot \prod_{q_i \in F_p}q_i^{\alpha_i}),(1/b+\varepsilon)\log_b(d \cdot \prod_{q_i \in F_p}q_i^{\alpha_i}))$. This implies that 
			
			\begin{align*}
			{f^2(n)} &\leq 2^t\cdot(2\cdots (b+1))^{(1/b+\varepsilon)\log_b(d \cdot \prod_{q_i \in F_p}q_i^{\alpha_i})}\\
			& = 2^t\cdot(b+1)!^{(1/b+\varepsilon)\log_b(\frac{1}{b^t}\prod_{i=0}^{b-1}(i+2)^{n_i})}\\
			& = \left(\frac{2}{(b+1)!^{(1/b+\varepsilon)}}\right)^t\cdot (b+1)!^{(1/b+\varepsilon)\sum_{i=0}^{b-1}n_i\log_b(i+2)}\\
			&\leq 1\cdot (b+1)!^{(1/b+\varepsilon)\log_b(b+1)\sum_{i=0}^{b-1}n_i}\\
			& \leq b^{c_b\sum_{i=0}^{b-1}n_i}\\
			& = c \cdot n^{c_b}
			\end{align*}			
			for some $c>0$, $0<c_b<1$, where the last inequality comes from the choice of $\varepsilon$.
			
			Suppose now, on the other hand, that $n_{b-1}<M$ and $n_{b-3}< M$. First, if $n_{b-2}<M$, as $M \leq \log_b(n)/4$ and $\sum_{i=0}^{b-1}n_i \leq \log_b{n}$, we have $\sum_{i=0}^{b-4}n_i \geq \log_b(n)/4$. Then,			
			\begin{align*}
			\frac{f(n)}{n} & =\frac{\prod_{i=0}^{b-1}(i+2)^{n_i}}{b^{\sum_{i=0}^{n-1}-1}}\\
			& \leq b \cdot \left(\frac{b-1}{b}\right)^{\sum_{i=0}^{b-3}n_i}\left(\frac{b+1}{b}\right)^{n_{b-1}}\\
			&\leq b\cdot  \left(\frac{b^2-1}{b^2}\right)^{\log_b(n)/4}\\
			&=b\cdot n^{\log_b\left(\frac{b^2-1}{b^2}\right)/4},
			\end{align*}
		as desired, since $\log_b\left(\frac{b^2-1}{b^2}\right)<0$.
		
		Finally, if $n_{b-1}, n_{b-3} < M \leq \log_b(n)/4$ and $n_{b-2} \leq M$, then, applying the trivial bound $f(m)\leq (b+1)^{1+\log_b{m}}$ and noting that $t \geq n_ {b-2}$ and $\sum_{i=0}^{n-4}n_i+n_{b-2} \geq \log_b(n)/2$, we get
		\begin{align*}
			\frac{f^2(n)}{n}&\leq\frac{2^t\cdot f(d \cdot \prod_{q_i \in F_p}q_i^{\alpha_i})}{b^{\sum_{i=0}^{b-1}n_i-1}}\\
			&\leq\frac{2^t \cdot(b+1)^{1+\log_b(d \cdot \prod_{q_i \in F_p}q_i^{\alpha_i})}}{b^{\sum_{i=0}^{b-1}n_i-1}}\\
			&=\frac{2^t \cdot(b+1)^{1+\log_b((\prod_{i=0}^{b-1}(i+2)^{n_i})/b^t)}}{b^{\sum_{i=0}^{b-1}n_i-1}}\\
			&\leq b(b+1)\cdot\left(\frac{2}{b+1}\right)^t\cdot\left(\frac{(b+1)^{\log_b(b-2)}}{b}\right)^{\sum_{i=0}^{b-4}n_i}\cdot\left(\frac{b+1}{b}\right)^{n_{b-2}}\\
			&\qquad \cdot\left(\frac{(b+1)^{\log_b(b-1)}}{b}\right)^{n_{b-3}}\cdot\left(\frac{(b+1)^{\log_b(b+1)}}{b}\right)^{n_{b-1}}\\
			&\leq b(b+1)\cdot\left(\frac{2}{b}\right)^{n_{b-2}}\cdot\left(\frac{(b+1)^{\log_b(b-2)}}{b}\right)^{\sum_{i=0}^{b-4}n_i} \cdot\left(\frac{(b+1)^{\log_b(b-1)}}{b}\right)^{n_{b-3}}\\
			&\qquad\cdot\left(\frac{(b+1)^{\log_b(b+1)}}{b}\right)^{n_{b-1}}.
		\end{align*}
		By the second item of Lemma \ref{lem:forthmt2blarge}, $\frac{(b+1)^{\log_b(b-1)}}{b} < 1$. This inequality, together with the simple fact that $(b+1)^{\log_b(b-2)} \geq 2$ for $b \geq 4$, implies 
		\begin{align*}
			\frac{f^2(n)}{n}&\leq b(b+1)\cdot\left(\frac{(b+1)^{\log_b(b-2)}}{b}\right)^{n_{b-2}+\sum_{i=0}^{b-4}n_i} \cdot\left(\frac{(b+1)^{\log_b(b+1)}}{b}\right)^{n_{b-1}}\\
			&\leq b(b+1) \left(\frac{(b+1)^{2\log_b(b-2)+\log_b(b+1)}}{b^3}\right)^{\log_b(n)/4},
		\end{align*}
		where we used that $n_{b-2}+\sum_{i=0}^{b-4} \geq \log_b(n)/2$ and $n_{b-1}\leq \log_b(n)/4$ in the last inequality. This completes the proof, since, by the third item of Lemma \ref{lem:forthmt2blarge}, the expression raised to $\log_b{n}$ in the last line above is smaller than $1$, whence $\frac{f^2(n)}{n}$ is bounded by $b(b+1)\cdot n^\gamma$ for some $\gamma<0$.
		\end{enumerate}
	\end{proof}	
	\begin{remark}\label{rem:loglog}
		The proof of Theorem \ref{thm:f23converges} gives that $S_{2,b}(n) \leq n^\gamma$ for some $\gamma<1$ if $n$ is large enough. As in Theorem \ref{thm:perst1}, this implies that the persistence of every number $n$ under $S_{2,b}$ is at most $c\log\log{n}$ for some constant $c$. It is not hard to see that, as before, there is some sequence of integers for which this is sharp up to a constant factor (i.e., there exists an increasing sequence $(n_k)_{k\geq 1}$ of integers such that the persistence of $n_k$ is at least $c'\log\log{n_k}$ for some $c'>0$ and every $k$).
	\end{remark}
	
	\subsection{The $(4,5)$ problem diverges}
	
	Using a proof similar to the proof of Theorem \ref{thm:f23converges}, one can show that the sequence of iterates of both $S_{3,4}$ and $S_{3,5}$ starting from every integer stabilizes (indeed, one can again prove that, in case $f=S_{3,4}$ or $f=S_{3,5}$, for every sufficiently large $n$, $f^j(n) \leq n$ for some $j \in \{1,2,3,4\}$). On the other hand, our next result shows that, assuming Conjecture \ref{conj:uniformdistr}, $S_{4,5}$ is the smallest instance where the opposite behavior occurs, namely the sequence of iterates starting from every sufficiently large integer diverges.	
	\begin{theorem}\label{thm:45diverges}
		Conjecture \ref{conj:uniformdistr} implies the following: there is an integer $n_0$ such that, for every $n \geq n_0$, the sequence of iterates $(S_{4,5}^k(n))_{k \geq 0}$ diverges to infinity.
	\end{theorem}	
	\begin{proof}
		Put $f=S_{4,5}$. We will prove the following statement which obviously implies the theorem: there is $n_0$ such that, for every $n \geq n_0$, $f^5(n) > n$.
		
		We apply Conjecture \ref{conj:uniformdistr} for $q=5$, $a=1$ and $F=\{2,3,7\}$ to get the following statement: for every $\varepsilon > 0$, there is $N$ such that $2^x\cdot 3^y \cdot 7^z$ is $\varepsilon$-equidistributed whenever one of $x$, $y$, $z$ is at least $N$. In particular, the number $4^{x}\cdot 6^{y} \cdot 7^{z}\cdot 8^{w}$ is equidistributed whenever one of $x$, $y$, $z$ and $w$ is at least $N$.
		
		Take $\varepsilon = 0.001$, put $\delta=1/5-\varepsilon$ and let $n \geq 5^{4M+M^2}$, with $M\delta^3(\log_5{4})^3 \geq N$, where $N$ is the integer given by the application of Conjecture \ref{conj:uniformdistr} as in the paragraph above with this value of $\varepsilon$; and $M$ is large enough as for the last inequality in (\ref{eq:bigm}) to hold.
		
		Let $a_0, \dots, e_0$ denote, respectively, $\#0(n), \dots, \#4(n)$; and, for $k \geq 1$, let $a_k, \dots, e_k$ denote, respectively, $\#0(4^{b_{k-2}+a_{k-1}}\cdot 6^{c_{k-1}} \cdot 7^{d_{k-1}}\cdot 8^{e_{k-1}}), \dots, \#4(4^{b_{k-2}+a_{k-1}}\cdot 6^{c_{k-1}} \cdot 7^{d_{k-1}}\cdot 8^{e_{k-1}})$, where we put $b_{-1}=0$. With this notation, we have		
		\begin{equation*}
		f^k(n)=4^{b_{k-2}+a_{k-1}}\cdot 5^{b_{k-1}} \cdot 6^{c_{k-1}} \cdot 7^{d_{k-1}} \cdot 8^{e_{k-1}}
		\end{equation*}
		for every $k \geq 1$.
		
		Assume first that one of $a_0$, $c_0$, $d_0$, $e_0$ is at least $M$. As $M \geq N/\delta^3(\log_5{4})^3 > N$, this implies that $4^{a_{0}}\cdot 6^{c_{0}} \cdot 7^{d_{0}}\cdot 8^{e_{0}}$ is $\varepsilon$-equidistributed. In particular, we have		
		\begin{align*}
		a_1,\dots,e_1 &\geq  
		\delta\log_5(4^{a_{0}}\cdot 6^{c_{0}} \cdot 7^{d_{0}}\cdot 8^{e_{0}})\\
		&=\delta(a_0\log_5{4}+c_0\log_5{6}+d_0\log_5{7}+e_0\log_5{8}).
		\end{align*} 
		
		In turn, as $a_1,\dots,e_1 \geq \delta (a_0+c_0+d_0+e_0)\log_5{4} > N$, this implies that $4^{b_{0}+a_{1}}\cdot 6^{c_1} \cdot 7^{d_1}\cdot 8^{e_{1}}$ is $\varepsilon$-equidistributed, so
		\begin{align*}
		a_2,\dots,e_2 &\geq 
		\delta\log_5(4^{b_0+a_{1}}\cdot 6^{c_{1}} \cdot 7^{d_{1}}\cdot 8^{e_{1}})\\
		&=\delta((b_0+a_1)\log_5{4}+c_1\log_5{6}+d_1\log_5{7}+e_1\log_5{8})\\
		&=\delta\log_5{4}\cdot b_0+\delta^2\log_5{1344}(a_0\log_5{4}+c_0\log_5{6}+d_0\log_5{7}+e_0\log_5{8}).
		\end{align*} 
		
		As the choice of $M$ guarantees that the $a_2,\dots,e_2$ and $a_3,\dots,b_3$ are greater than $N$, the same reasoning can be applied two more times to get that		
		\begin{align*}
		a_3,\dots,e_3 &\geq
		\delta\log_5(4^{b_1+a_{2}}\cdot 6^{c_{2}} \cdot 7^{d_{2}}\cdot 8^{e_{2}})\\
		&=\delta((b_1+a_2)\log_5{4}+c_2\log_5{6}+d_2\log_5{7}+e_2\log_5{8})\\
		&\geq 
		\delta^2\log_5{1344}\log_5{4}\cdot b_0\\
		&\quad+\delta^2(\log_5{4}+\delta(\log_5{1344})^2)(a_0\log_5{4}+c_0\log_5{6}+d_0\log_5{7}+e_0\log_5{8})
		\end{align*} 
		and
		\begin{align*}
		a_4,\dots,e_4 &\geq
		\delta\log_5(4^{b_2+a_{3}}\cdot 6^{c_{3}} \cdot 7^{d_{3}}\cdot 8^{e_{3}})\\
		&=\delta((b_2+a_3)\log_5{4}+c_3\log_5{6}+d_3\log_5{7}+e_3\log_5{8})\\
		&\geq \delta^2((\log_5{4})^2+\delta\log_5{4}(\log_5{1344})^2)b_0\\
		&\quad + \delta^3\log_5{1344}(2\log_5{4}+\delta(\log_5{1344})^2)\cdot \\
		&\quad\cdot(a_0\log_5{4}+c_0\log_5{6}+d_0\log_5{7}+e_0\log_5{8}).
		\end{align*}
		
		Finally, this implies that 
		
		\begin{align*}
		\frac{f^5(n)}{n}& > \frac{4^{b_{3}+a_{4}}\cdot 5^{b_{4}} \cdot 6^{c_{4}} \cdot 7^{d_{4}} \cdot 8^{e_{4}}}{5^{a_0+b_0+c_0+d_0+e_0}} \\
		& \geq \left(\frac{4^{\delta^2\log_5{4}(\log_5{4}+\delta(\log_5{1344})^2)}\cdot 6720^{\delta^3\log_5{4}\log_5{1344}(2\log_5{4}+\delta(\log_5{1344})^2)}}{5}\right)^{a_0+c_0+d_0+e_0}\cdot\\
		&\quad\cdot\left(\frac{4^{\delta^2\log_5{1344}\log_5{4}}\cdot 6720^{\delta^2\log_5{4}(\log_5{4}+\delta(\log_5{1344})^2)}}{5}\right)^{b_0}\\
		& > 1,
		\end{align*}
		as a straightforward computation shows that each of the expressions inside the parenthesis are greater than $1$ (indeed, the first and the second expression are greater than $1.14$ and $1.06$, respectively).	
		
		Suppose now, on the other hand, that each of $a_0$, $c_0$, $d_0$, $e_0$ is less than $M$. This implies that $b_0 > M^2 > N$, which in turn implies that $4^{b_0+a_1}\cdot 6^{c_1}\cdot 7^{d_1}\cdot 8^{e_1}$ is $\varepsilon$-equidistributed. Hence, we have $a_2,\cdots,e_2 \geq \delta\log_5(4^{b_0+a_1}\cdot 6^{c_1}\cdot 7^{d_2}\cdot 8^{e_1}) \geq (\delta\log_5{4})b_0>N$. Again, this implies that $4^{b_1+a_2}\cdot 6^{c_2}\cdot 7^{d_2}\cdot 8^{e_2}$ is $\varepsilon$-equidistributed and $a_3,\dots,e_3 \geq \delta\log_5(4^{b_1+a_2}\cdot 6^{c_2}\cdot 7^{d_2}\cdot 8^{e_2}) \geq \delta^2\log_5{4}\log_5{1344}\cdot b_0 > N$. Finally, this implies that $a_4,\dots,e_4 \geq \delta\log_5(4^{b_2+a_3}\cdot 6^{c_3}\cdot 7^{d_3}\cdot 8^{e_3}) \geq \delta^2\log_5{4}(\log_5{4}+\delta(\log_5{1344})^2)b_0$, and then		
		\begin{align*}
		\frac{f^5(n)}{n}& > \frac{4^{b_3+a_4}\cdot 5^{b_4}\cdot 6^{c_4}\cdot 7^{d_4} \cdot 8^{e_ 4}}{5^{a_0+b_0+c_0+d_0+e_0}}\\
		&> \left(\frac{4^{\delta^2\log_5{4}\log_5{1344}}\cdot 6720^{\delta^2\log_5{4}(\log_5{4}+\delta(\log_5{1344})^2)}}{5}\right)^{b_0}\cdot \left(\frac{1}{5}\right)^{4M}\\
		&\geq \left(\frac{4^{\delta^2\log_5{4}\log_5{1344}}\cdot 6720^{\delta^2\log_5{4}(\log_5{4}+\delta(\log_5{1344})^2)}}{5}\right)^{M^2}\cdot \left(\frac{1}{5}\right)^{4M}\\
		& > 1,\stepcounter{equation}\tag{\theequation}\label{eq:bigm}
		\end{align*}
		for large $M$, as the expression being raised to $M^2$ is greater than $1$ (approximately $1.06$).		
	\end{proof}

	\subsection{Larger $t$ and $b$}
	
	In this section, we consider the behavior of the $t$-shifted problem in base $b$ when $t$ is bounded by a function of $b$. As one would expect, the equidistribution conjectures imply that, if $b$ is very large compared to $t$, then the sequence of iterates starting from any integer stabilizes in the $t$-shifted problem in base $b$. One the other hand, if $t$ is very close to $b$, then almost no sequence stabilizes. Our next results give some estimates on the ranges of $t$ and $b$ (both for small and large $b$) where each of those behaviors appear. Again, we start with a technical lemma.	
	\begin{lemma}\label{lem:forthmtb/4}
		Let $t$ and $b$ be positive integers such that $b\geq 5$ and $t \leq b/4$. Then
		\leavevmode
		\begin{enumerate}[(i)]
			\item $b^{\log_b(t)-1}\cdot\frac{(b+t-1)^{\log_b(b+t-1)}}{b}<1$;

			\item $\left(\frac{(b+t-1)!}{(t-1)!}\right)^{\frac{1}{b}\log_b(b+t-1)} < b.$
		\end{enumerate}
	\end{lemma}
	
	\begin{proof}
		\begin{enumerate}[(i)]
			\item The inequality is equivalent, applying logarithms, to 			
			\begin{equation}\label{eq:firstequiv}
			\log{b}\cdot\log{t}+(\log(b+t-1))^2<2(\log{b})^2.
			\end{equation}
			The left-hand side of (\ref{eq:firstequiv}) is increasing with $t$, so bounded from above by $\log{b}\cdot\log(b/4)+(\log(5b/4))^2$. Expanding this expression, we get			
			\begin{align*}
			2(\log{b})^2+\log(25/64)\log{b}+(\log(5/4))^2,
			\end{align*}
			which is smaller than $2(\log{b})^2$ whenever $b	\geq e^{-\log(5/4)^2/\log(25/64)}\approx 1.05$.
				
			\item The inequality is equivalent, taking logarithms twice, to			
			\begin{equation}\label{eq:bound2equiv}
			\log\log(b+t-1)+\log\log\left(\frac{(b+t-1)!}{(t-1)!}\right)<\log{b}+2\log\log{b}.
			\end{equation}			
			By the inequality of arithmetic and geometric means, we have that			
			\begin{align*}
			\frac{(b+t-1)!}{(t-1)!}&=t(t+1)\cdots(b+t-1)\\&\leq \left(\frac{t+(t+1)+\cdots+(b+t-1)}{b}\right)^b\\
			&=\left(\frac{b+2t-1}{2}\right)^b\\
			&<\left(\frac{3b}{4}\right)^b,
			\end{align*} 
			as $t \leq b/4$.
			
			We also have $b+t-1< 5b/4$. Plugging these two estimates on the left-hand side of (\ref{eq:bound2equiv}) and using that $\log\log{x}$ is a concave and increasing function on its domain, we get that		
			\begin{align*}
			\log\log(b+t-1)+\log\log\left(\frac{(b+t-1)!}{(t-1)!}\right) &< \log\log(5b/4)+\log\log(3b/4)^b\\
			&=\log{b}+\log\log(5b/4)+\log\log(3b/4)\\
			&<\log{b}+2\log\log\left(\frac{5b/4+3b/4}{2}\right)\\
			& = \log{b}+2\log\log{b}.
			\end{align*}
		\end{enumerate}
	\end{proof}	
	\begin{theorem}\label{thm:b/4converges}
		Conjecture \ref{conj:uniformdistr} implies the following: for every prime number $b \geq 5$ and $t \leq b/4$, the sequence of iterates $(S_{t,b}^k(n))_{k \geq 1}$ stabilizes for every positive integer $n$.
	\end{theorem}	
	\begin{proof}
		The proof follows the ideas of the proof of Theorem \ref{thm:f23converges}. Let $t$ and $b$ be integers such that $3 \leq t \leq b/4$ (the cases $t=1$ and $t=2$ were covered by Theorems \ref{thm:f1bconverges} and \ref{thm:f23converges}, respectively). Again, we will prove that, for every sufficiently large integer $n$, either $f(n) \leq n^{c_b}$ or $f^2(n) \leq n^{c_b}$, where $f=S_{t,b}$ and $0<c_b<1$.
		
		We apply Conjecture \ref{conj:uniformdistr} with $q=b$, $a=1$, $F=\{p \text{ prime}:p \leq b+t-1, p \neq b\}$ to get the following statement: 
		for every $\varepsilon > 0$, there is $N$ such that the number $\prod_{p_i \text{ prime}:p_i \leq b+t-1, p_i \neq b}p_i^{a_i}$ is $\varepsilon$-equidistributed if any of the $a_i$ is at least $N$. 
		Taking the maximum of the $N$ obtained by each application of the conjecture, we get the following statement: 
		for every $\varepsilon > 0$, there is $N$ such that, for every proper divisor $d$ of $b$ and every prime divisor $p$ of $b$, the number $d\prod_{q_i \text{ prime}, q_i \leq b+t-1, q_i \neq p}q_i^{a_i}$ is $\varepsilon$-equidistributed if any of the $a_i$ is at least $N$.
		
		Let $\varepsilon> 0$ be so small as to satisfy 
		\begin{equation*}
			\frac{1}{b}\cdot\left(\frac{(b+t-1)!}{(t-1)!}\right)^{(\frac{1}{b}+\varepsilon)\log_b(b+t-1)}<1,
		\end{equation*}
		 which is possible by the second item of Lemma \ref{lem:forthmtb/4}. Moreover, let $n \geq b^{2(b-1)M}$, with $M \geq N$, where $N$ is as in the paragraph above. For $i \in \{0,\dots,b-1\}$, let $n_i$ denote the number of digits $i$ in the base-$b$ expansion of $n$. We know that $\sum_{i=0}^{b-1}n_i \geq \log_b{n} \geq 2(b-1)M$. This implies that either $n_i \geq M$ for some $i\neq b-t$ or $n_{b-t} \geq (b-1)M$.
		
		By the definition of $f$, we have $f(n)=\prod_{i=0}^{b-1}(i+t)^{n_i}$, where $n_i$ is the number of digits $i$ of $n$ in base $b$. Note that $\prod_{i=0}^{b-1}(i+t)^{n_i}=b^{n_{b-t}}\cdot\prod_{i=0,i\neq b-t}^{b-1}(i+t)^{n_i}$, and that $b$ does not divide the second product. Then, we have $f^2(n)=t^{n_{b-t}}\cdot f(\prod_{i=0,i\neq b-t}^{b-1}(i+t)^{n_i})$
		
		Suppose first that $n_i \geq M$ for some $i\neq b-t$. This implies that the number $\prod_{i=0,i\neq b-t}^{b-1}(i+t)^{n_i}$ is $\varepsilon$-equidistributed, i.e., the number of occurrences of every digit from $0$ to $b-1$ belongs to the interval $((1/b-\varepsilon)\log_b(\prod_{i=0,i\neq b-t}^{b-1}(i+t)^{n_i}),(1/b+\varepsilon)\log_b(\prod_{i=0,i\neq b-t}^{b-1}(i+t)^{n_i}))$. Hence, as $n \geq b^{\sum_{i=0}^{b-1}n_i-1}$, it follows that 		
		\begin{align*}
		\frac{f^2(n)}{n} &\leq \frac{1}{n} \cdot t^{n_{b-t}}\cdot(t\cdots (t+b-1))^{(\frac{1}{b}+\varepsilon)\log_b(\prod_{i=0,i\neq b-t}^{b-1}(i+t)^{n_i})}\\
		& = \frac{1}{n} \cdot b^{n_{b-t}\log_b{t}}\cdot\left(\frac{(b+t-1)!}{(t-1)!}\right)^{(\frac{1}{b}+\varepsilon)\sum_{i=0,i\neq b-t}^{b-1}{n_i}(\log_b(i+t))}\\
		& \leq \frac{1}{n} \cdot b^{n_{b-t}\log_b{t}}\left(\left(\frac{(b+t-1)!}{(t-1)!}\right)^{(\frac{1}{b}+\varepsilon)\log_b(b+t-1)}\right)^{\sum_{i=0,i\neq b-t}^{b-1}n_i}\\
		& \leq b \cdot \left(b^{\log_b{t}-1}\right)^{n_{b-t}}\left(\frac{1}{b}\cdot\left(\frac{(b+t-1)!}{(t-1)!}\right)^{(\frac{1}{b}+\varepsilon)\log_b(b+t-1)}\right)^{\sum_ {i=0,i\neq b-t}^{b-1}n_i}.\stepcounter{equation}\tag{\theequation}\label{eq:bound}
		\end{align*}
		
		The choice of $\varepsilon$ implies that the expression inside the second parenthesis in the last line of (\ref{eq:bound}) is $b^{\gamma}$ for some $\gamma < 0$, which completes the proof of this case, as $\log_b(t)-1<0$ and then (\ref{eq:bound}) is bounded from above by $b\cdot b^{\gamma'\sum_{i=0}^{b-1}n_i}\leq b\cdot n^{\gamma'}$ for some $\gamma' < 0$.
		
		Suppose now, on the other hand, that $n_i < M$ for every $i\neq b-t$. This implies, as $\sum_{i=0}^{b-1}n_i \geq 2(b-1)M$, that $n_{b-t}\geq (b-1)M\geq\sum_{i=0,i\neq b-t}^{b-1}n_i $, and hence $n_{b-t}\geq \log_b(n)/2 \geq \sum_{i=0,i\neq b-t}^{b-1}n_i$. Applying a trivial bound $f(m) \leq (b+t-1)^{1+\log_b{m}}$, we get that		
		\begin{align*}
		\frac{f^2(n)}{n}&=\frac{1}{n}\cdot t^{n_{b-t}}f\left(\prod_{i=0,i\neq b-t}^{b-1}(i+t)^{n_i}\right)\\
		& \leq  \frac{1}{n}\cdot b^{n_{b-t}\log_b{t}}\cdot (b+t-1)^{1+\sum_{i=0}^{b-1}\log_b(i+t)n_i}\\
		& \leq 2b^2\cdot b^{(\log_b{t}-1)n_{b-t}}\cdot \left(\frac{(b+t-1)^{\log_b(b+t-1)}}{b}\right)^{\sum_{i=0,i\neq b-t}^{b-1}n_i}\\
		& \leq 2b^2 \cdot \left(b^{\log_b(t)-1}\cdot\frac{(b+t-1)^{\log_b(b+t-1)}}{b}\right)^{\log_b(n)/2}\\
		&=2b^2\cdot n^\gamma,
		\end{align*}
		where, by the first item of Lemma \ref{lem:forthmtb/4}, $b^{\log_b(t)-1}\cdot\frac{(b+t-1)^{\log_b(b+t-1)}}{b}<1$, whence $\gamma < 0$. This concludes the proof.		
	\end{proof}
	
	The estimates in the proof of Theorem \ref{thm:b/4converges} can be applied asymptotically in $b$ (instead of for every $b \geq 4$) using Stirling's formula (instead of a precise inequality for every $n$)  to improve the constant $1/4$ in to approximately $0.316$ for large $b$. Namely, the following result, whose proof is very similar to the proof of Theorem \ref{thm:b/4converges} and will be omitted, holds:	
	\begin{theorem}\label{thm:tsmallconverges}
		Conjecture \ref{conj:uniformdistr} implies the following: let $c_0$ be the solution of $2\log(1+c)+c\log(1+1/c)=1$ in the interval $(0,1)$ ($c_0 \approx 0.315999$). Then, for every $c \leq c_0$, there is $b_0$ with the following property: for every prime number $b \geq b_0$ and $t \leq c b$, the sequence of iterates $(S_{t,b}^k(n))_{k \geq 1}$ stabilizes for every positive integer $n$.
	\end{theorem}	
	\begin{remark}
	What we stated in 
		Remark \ref{rem:loglog} holds for Theorems \ref{thm:b/4converges} and \ref{thm:tsmallconverges} as well, i.e., the persistence is bounded by $c\log\log{n}$ for every $n$ and some $c$, and there is an increasing sequence of integers $(n_k)_{k\geq 1}$ with persistence at least $c'\log\log{n_k}$ for some $c'>0$ and every~$k$.
	\end{remark}
	
	On the other end of the spectrum, we have a divergence result, which we state now, after a technical lemma.	
	\begin{lemma}\label{lem:forthmtlarge} Let $c_0$ be the solution $2\log{c}+\log(c+1)+c\log(1+1/c) = 1$ in the interval $(0,1)$ ($c_0 \approx 0.865722$). Then, for every $c > c_0$, there is $b_0$ with the following property: Let $t$ and $b$ be integers, with $b \geq b_0$ and $t \geq cb$. If we put $\delta = 1/b-1/b^2$, then the following inequalities hold:
		\leavevmode
		\begin{enumerate}[(i)]
			\item $\left(\frac{(t-1+b)!}{(t-1)!}\right)^{\delta\log_b{t}} > b;$
			
			\item $t^{\delta\log_b{t}} \cdot \left(\frac{(t-1+b)!}{(t-1)!}\right)^{(b-2)\delta^2(\log_b{t})^2} > b.$
		\end{enumerate}
	\end{lemma}	
	\begin{proof}
		We will prove that the logarithm to base $b$ of the two expressions on the left-hand sides of \textit{i)} and \textit{ii)} are greater than $1$ for large $b$. For that purpose, we use the following logarithmic form of the well-known Stirling approximation: $$\displaystyle \log{n!} = n\log{n}-n+O(\log{n}).$$ 
		
		\begin{enumerate}[(i)] 
			\item We have 
			\begin{align*}
			\log_b&\left(\frac{(t-1+b)!}{(t-1)!}\right)^{\delta\log_b{t}} =\\ &= \delta\frac{\log{t}}{(\log{b})^2} \cdot\big((t-1+b)\log(t-1+b)-(t-1)\log(t-1)\\
			&\quad-b+O(\log{b})\big) \\
			& = \delta\frac{\log{t}}{(\log{b})^2} \cdot(b\log(t-1+b)+(t-1)\log(1+b/(t-1))\\
			&\quad-b+O(\log{b}))\\
			&\geq \delta\frac{\log{cb}}{(\log{b})^2}\cdot(b\log((c+1)b)+cb\log(1+1/c)-b+O(\log{b}))\\
			&\geq \delta b\left(1+\frac{\log{c}+\log(c+1)+c\log(1+1/c)-1}{\log{b}}+O\left(\frac{1}{b\log{b}}\right)\right)\\
			&\geq 1+\frac{\log{c}+\log(c+1)+c\log(1+1/c)-1}{\log{b}}+O\left(\frac{1}{b}\right),
			\end{align*}
			since $\delta b=1-1/b$. This expression is greater than $1$ for large $b$ as we have $\log{c}+\log(c+1)+c\log(1+1/c)-1> \log{c_0}+\log(c_0+1)+c_0\log(1+1/c_0)-1 > 2\log{c_0}+\log(c_0+1)+c_0\log(1+1/c_0)-1 = 0$.
			
			\item We apply the bound obtained in the first part of the proof to get
			\begin{align*}
			\log_b&\left(t^{\delta\log_b{t}} \cdot \left(\frac{(t-1+b)!}{(t-1)!}\right)^{(b-2)\delta^2(\log_b{t})^2}\right) \\ &\quad= \delta\left(\frac{\log{t}}{\log{b}}\right)^2+(b-2)\delta\frac{\log{t}}{\log{b}}\log_ b\left(\frac{(t-1+b)!}{(t-1)!}\right)^{\delta\log_b{t}}\\
			&\quad\geq \left(\frac{1}{b}-\frac{1}{b^2}\right)\left(1+\frac{\log{c}}{\log{b}}\right)^2+\left(1-\frac{2}{b}\right)\left(1+\frac{1}{b}\right)\left(1+\frac{\log{c}}{\log{b}}\right)\cdot\\
			&\quad\quad\cdot\left(1+\frac{\log{c}+\log(c+1)+c\log(1+1/c)-1}{\log{b}}+O\left(\frac{1}{b}\right)\right)\\
			&\quad=1+\frac{2\log{c}+\log(c+1)+c\log(1+1/c)-1}{\log{b}}+O\left(\frac{1}{b}\right),
			\end{align*}
			which is greater than $1$ for large $b$ since $2\log{c}+\log(c+1)+c\log(1+1/c)-1> 2\log{c_0}+\log(c_0+1)+c_0\log(1+1/c_0)-1 = 0$.
		\end{enumerate}
	\end{proof}	
	\begin{theorem}\label{thm:divergencelarge}
		Conjecture \ref{conj:uniformdistr} implies the following: let $c_0$ be the solution $2\log{c}+\log(c+1)+c\log(1+1/c) = 1$ in the interval $(0,1)$ ($c_0 \approx 0.865722$). Then, for every $c > c_0$, there is $b_0$ with the following property: for every prime number $b \geq b_0$ and positive integer $t \geq c b$, the sequence of iterates $(S_{t,b}^k(n))_{k \geq 1}$ diverges for every sufficiently large integer $n$.
	\end{theorem}	
	\begin{proof}
		Fix $c>c_0$. Let $b_0$ be the integer given by Lemma \ref{lem:forthmtlarge} for this $c$, and let $t$, $b$ be integers such that $t \geq cb$ and $b \geq b_0$.
		
		We apply Conjecture \ref{conj:uniformdistr} with $q=b$, $a=1$, $F=\{p \text{ prime}:p \leq b+t-1, p \neq b\}$ to get the following statement: 
		for every $\varepsilon > 0$, there is $N$ such that the number $\prod_{p_i \text{ prime}:p_i \leq b+t-1, p_i \neq b}p_i^{a_i}$ is $\varepsilon$-equidistributed if any of the $a_i$ is at least $N$.
		
		Put $f=S_{t,b}$. We can write $f(n)= \prod_{i=0}^{b-1} (i+t)^{n_i} = b^{n_{b-t}} \cdot \prod_{i=0, i \neq b-t}^{b-1} (i+t)^{n_i}$. As $b$ is a prime number, $b \nmid \prod_{i=0, i \neq b-t}^{b-1} (i+t)^{n_i}$, and then $f^2(n) = t^{n_{b-t}}\cdot f(\prod_{i=0, i \neq b-t}^{b-1} (i+t)^{n_i})$. Let $n'_i$ denote the number of digits $i$ in $\prod_{i=0, i \neq b-t}^{b-1} (i+t)^{n_i}$. Then we may rewrite $f^2(n) = t^{n_{b-t}}\cdot \prod_{i=0}^{b-1}(i+t)^{n'_i}=b^{n'_{b-t}}\cdot t^{n_{b-t}} \cdot \prod_{i=0,i\neq b-t}^{b-1} (i+t)^{n'_{i}}$.
		
		Let $\varepsilon = 1/b^2$ and put $\delta=1/b-\varepsilon$. Let $M \geq N/(\delta \log_b{2})$, where $N$ is the integer given by the application of Conjecture \ref{conj:uniformdistr} as above with $\varepsilon = 1/b^2$, and $M$ is large enough as to satisfy that (\ref{eq:condition}) is greater than $1$. Let $n \geq b^{(b-1)M+M^2}$ be an integer. We will prove that $f^3(n) > n$, which implies that the sequence of iterates of $f$ starting from any integer at least $b^{(b-1)M+M^2}$ diverges.		
		
		Either $n_i \geq M$ for some $i \neq b-t$ or $n_{b-t} \geq M^2 \geq M \geq N$. In the first case, this implies that $n'_i \geq \delta\log_b(\prod_{i=0, i \neq b-t}^{b-1} (i+t)^{n_i}) \geq (\delta \log_b{2})\sum_{i=0, i \neq b-t}^{b-1} n_i \geq (\delta \log_b{2})M \geq N$. In any case, the number $t^{n_{b-t}} \cdot \prod_{i=0,i\neq b-t} (i+t)^{n'_{i}}$ is $\varepsilon$-equidistributed. This means that every digit from $0$ to $b-1$ appears at least $\delta\log_b(t^{n_{b-t}} \cdot \prod_{i=0,i\neq b-t} (i+t)^{n'_{i}}) = \delta(n_{b-t}\log_b{t}+\sum_{i=0,i=b-t}^{b-1}n'_i\log_b(i+t))$ times in this number, and hence
		\begin{align*}
		f^3(n) & = f\bigg(b^{n'_{b-t}}\cdot t^{n_{b-t}} \cdot \prod_{i=0,i\neq b-t}^{b-1} (i+t)^{n'_{i}}\bigg)\\
		& = t^{n'_{b-t}}\cdot f\bigg(t^{n_{b-t}} \cdot \prod_{i=0,i\neq b-t}^{b-1} (i+t)^{n'_{i}}\bigg)\\
		&\geq t^{n'_{b-t}}(t(t+1)\dots(b+t-1))^{\delta(n_{b-t}\log_b{t}+\sum_{i=0,i \neq b-t}^{b-1}n'_i\log_b(i+t))}\\
		&\geq t^{n'_{b-t}}\cdot\left(\frac{(t-1+b)!}{(t-1)!}\right)^{\delta(n_{b-t}+\sum_{i=0,i\neq b-t}^{b-1}n'_i)\log_b{t}}\tag{\theequation}\label{eq:lowerboundf3}.
		\end{align*}
		
		Suppose first that  $n_i \geq M$ for some $i \neq b-t$. Then $\prod_{i=0, i \neq b-t}^{b-1} (i+t)^{n_i}$ is $\varepsilon$-equidistributed, which implies that $n'_i \geq \delta\log_b(\prod_{i=0, i \neq b-t}^{b-1} (i+t)^{n_i}) \geq \delta\log_b{t}\sum_{i=0, i \neq b-t}^{b-1} n_i$ for every $1 \leq i \leq b-1$. In this case, bound (\ref{eq:lowerboundf3}) implies, together with Lemma \ref{lem:forthmtlarge}, that
		
		\begin{align*}
		f^3(n) &\geq\left(t^{\delta\log_b{t}} \cdot \left(\frac{(t-1+b)!}{(t-1)!}\right)^{(b-2)\delta^2(\log_b{t})^2}\right)^{\sum_{i=0,i\neq b-t}^{b-1}n_i}\cdot \left(\frac{(t-1+b)!}{(t-1)!}\right)^{(\delta\log_b{t}) n_{b-t}}\\
		& > b^{\sum_{i=0,i\neq b-t}^{b-1}n_i+n_{b-t}}\\
		& \geq n.
		\end{align*}
		
		Finally, if $n_i < M$ for every $i \neq b-t$, then $n_{b-t} \geq M^2$. Moreover, we have $n'_i \leq \log_b(\prod_{i=0, i \neq b-t}^{b-1} (i+t)^{n_i}) < b\log_b(2b)M$ for every $0 \leq i \leq b-1$, and then, using (\ref{eq:lowerboundf3}), we get that
		\begin{align*}
		\frac{f^3(n)}{n} &\geq \frac{1}{n} \cdot t^{n'_{b-t}}\cdot\left(\frac{(t-1+b)!}{(t-1)!}\right)^{\delta(n_{b-t}+\sum_{i=0,i\neq b-t}^{b-1}n'_i)\log_b{t}} \\
		& \geq \frac{1}{n}\cdot \left(\frac{(t-1+b)!}{(t-1)!}\right)^{(\delta\log_b{t})n_{b-t}}\\
		& > \left(\frac{1}{b}\cdot\left(\frac{(t-1+b)!}{(t-1)!}\right)^{\delta\log_b{t}}\right)^{n_{b-t}}\cdot b^{-\sum_{i=0,i\neq b-t}^{b-1}n_i}\\
		& \geq \left(\frac{1}{b}\cdot\left(\frac{(t-1+b)!}{(t-1)!}\right)^{\delta\log_b{t}}\right)^{M^2} \cdot b^{-bM} \stepcounter{equation}\tag{\theequation}\label{eq:condition} \\
		&> 1,
		\end{align*}
		if $M$ is large enough (since, by Lemma \ref{lem:forthmtlarge}, the base of $M^2$ in the last line of (\ref{eq:condition}) is greater than $1$). This concludes the proof.
		
	\end{proof}	
	\begin{remark}
		As we saw before (in Theorems \ref{thm:b/4converges} and \ref{thm:tsmallconverges}), if one applies bounds for $n!$ that hold for every $n$ instead of the asymptotic Stirling formula, it is possible to get a result of the following form, with $c$ being a constant greater than $c_0$ in Theorem \ref{thm:divergencelarge}: let $b \geq 7$ be a prime and $t \geq cb$. Conjecture \ref{conj:uniformdistr} implies that, for every sufficiently large integer $n$, the sequence of iterates $(S_{t,b}^k(n))_{k\geq 1}$ diverges to infinity.
	\end{remark}

	\section{Concluding remarks}
	
	As we made clear from the beginning, most of the main results in the present paper (Theorems \ref{thm:erdosbase3}, \ref{thm:t1b3pers}, \ref{thm:f23converges}, \ref{thm:45diverges}, \ref{thm:b/4converges} and \ref{thm:divergencelarge}) are conditional on the validity of Conjecture \ref{conj:equidist} or \ref{conj:uniformdistr}. There is ample experimental evidence in favor of Conjecture \ref{conj:equidist}, and further support would be provided if one could prove our main results unconditionally. Conversely, of course, if any one of these unconditional statements were proved to be false, then Conjecture \ref{conj:equidist} (or \ref{conj:uniformdistr}) would have to be false. However, given the computational evidence and robust heuristics available \cite{de2014sloane}, we feel confident that Conjectures \ref{conj:equidist} and  \ref{conj:uniformdistr} must be true.
	
	\section{Acknowledgments}
	
	We wish to thank the referee for several remarks and suggestions that led to a considerable improvement of our exposition.
	
	\bibliographystyle{jis}
	\bibliography{ref}
	
	\bigskip
	\hrule
	\bigskip
	
	\noindent 2010 {\it Mathematics Subject Classification}:
	Primary 	11A63; Secondary 11A67.

	\noindent \emph{Keywords: }persistence of a number, Sloane's conjecture, products of digits, shifted Sloane problem, equidistribution of digits.
	
	\bigskip
	\hrule
	\bigskip
	
	\noindent (Concerned with sequences
	\seqnum{A036461},
	\seqnum{A335808},
	\seqnum{A335824}, and
	\seqnum{A031346}.
	)
	
	\bigskip
	\hrule
	\bigskip

	\noindent
	Return to
	\htmladdnormallink{Journal of Integer Sequences home page}{http://www.cs.uwaterloo.ca/journals/JIS/}.
	\vskip .1in

\end{document}